\documentclass{birkjour}
 \newtheorem{thm}{Theorem}[section]
 
 \newtheorem{lemma}[thm]{Lemma}
 \newtheorem{prop}[thm]{Proposition}
 \theoremstyle{definition}
 
 \theoremstyle{remark}
 \newtheorem{rmk}[thm]{Remark}
 
 \numberwithin{equation}{section}

\usepackage{amsmath}
\usepackage{amsfonts}
\usepackage{amssymb}
\usepackage[english]{babel}
\usepackage{color}
\usepackage{graphicx}
\usepackage{lmodern}
\usepackage{amsthm}
\usepackage{mathrsfs}
\usepackage{microtype}
\usepackage{mathscinet}
\usepackage{enumitem}
\usepackage{hyperref}
\hypersetup{hidelinks}

\DeclareMathOperator\supp{supp}

\DeclareMathOperator\diver{div}

\DeclareMathOperator\Fuj{Fuj}
\DeclareMathOperator\GN{GN}

\begin{document}

%
%
%
%
%
%
%
%
%

\title[Scale-invariant wave equations in exponentially weighted spaces]{Global existence of solutions for semi-linear wave equation with scale-invariant damping and mass in exponentially weighted spaces}




\author[A. Palmieri]{Alessandro Palmieri}

\address{%
Faculty of Mathematics and Informatics\\
Institute of Applied Analysis\\
Technical University Bergakademie Freiberg\\
Pr\"{u}ferstra{\ss}e 9\\
09596  Freiberg \\
Germany}

\email{alessandro.palmieri.math@gmail.com}

\thanks{The PhD study of the author is supported by S\"{a}chsiches Landesgraduiertenstipendium.}


\date{January 1, 2004}

\begin{abstract}
In this paper we consider the following Cauchy problem for the semi-linear wave equation with scale-invariant  dissipation and mass and power non-linearity:
\begin{align}\label{CP abstract}
\begin{cases} u_{tt}-\Delta u+\dfrac{\mu_1}{1+t} u_t+\dfrac{\mu_2^2}{(1+t)^2}u=|u|^p, \\ u(0,x)=u_0(x), \,\, u_t(0,x)=u_1(x),
\end{cases}\tag{$\star$}
\end{align}
where $\mu_1, \mu_2^2$ are nonnegative constants and $p>1$.

 On the one hand we will prove a global (in time) existence result for \eqref{CP abstract} under suitable assumptions on the coefficients $\mu_1, \mu_2^2$ of the damping and the mass term and on the exponent $p$, assuming the smallness of data in exponentially weighted energy spaces. 
 On the other hand a blow-up result for \eqref{CP abstract} is proved for values of $p$ below a certain threshold, provided that the data satisfy some integral sign conditions.  
 
 Combining these results we find the critical exponent for \eqref{CP abstract} in all space dimensions under certain assumptions on $\mu_1$ and $\mu_2^2$.
 Moreover, since the global existence result is based on a contradiction argument, it will be shown firstly a local (in time) existence result.  
\end{abstract}

\maketitle

\section{Introduction}\label{Intro}

In this paper we study the global in time existence of small data solutions and the blow-up in finite time of solutions to the Cauchy problem 
\begin{align}\label{CP semilinear 1}
\begin{cases}
u_{tt}-\Delta u+\frac{\mu_1}{1+t} u_t+\frac{\mu_2^2}{(1+t)^2}u=|u|^p, & t>0, \,\, x\in \mathbb{R}^n, \\ u(0,x)=u_0(x), & x\in\mathbb{R}^n, \\ u_t(0,x)=u_1(x), & x\in\mathbb{R}^n,
\end{cases}
\end{align}  in any space dimension $n\geq 1$, where $\mu_1>0$ and $\mu_2^2\geq 0$ are constants and $p>1$.

The main purpose of the present article is to extend a result from \cite{NunPalRei16} to any spatial dimension $n\geq 1$. More precisely, considering the quantity
\begin{align}\label{delta definition}
\delta := (\mu_1-1)^2-4\mu_2^2
\end{align} which describes somehow the interplay between the damping and the mass term, in \cite[Theorem 2.2]{NunPalRei16} it is proved a global existence result of small data energy solutions in space dimensions $n=1,2,3,4$ for a certain range of $\delta$, assuming additional $L^1$ regularity for the initial data. The restriction on the spatial dimension is due to the employment of Gagliardo-Nirenberg inequality in the estimates of the non-linear term, which implies the restrictions $p\geq 2$ and $p\leq p_{\GN}(n):=\frac{n}{n-2}$ when $n\geq 3$ for the exponent of the non-linearity. In order to avoid this type of conditions on $p$, we may consider stronger assumptions on initial data. More specifically, we will consider data in energy spaces with weight of exponential type. 

Let us clarify the role of $\delta$ in the description of the equation 
\begin{align}\label{scale invariant wave equation with damp and mass}
u_{tt}-\Delta u+\frac{\mu_1}{1+t}u_t+\frac{\mu_2^2}{(1+t)^2}u=0.
\end{align}
Considering  $v(t,x)=(1+t)^\gamma u(t,x)$, we can transform \eqref{scale invariant wave equation with damp and mass} in a wave equation with either just a scale-invariant damping or just a scale-invariant mass with a suitable choice of $\gamma\in \mathbb{R}$, depending on the value of $\delta$.

In particular for $\delta\geq (n+1)^2$ we can transform \eqref{scale invariant wave equation with damp and mass} in a scale-invariant wave equation with a damping term, which looks like effective under the point of view of decay estimates (even though, according to the classification introduced in \cite{WiThe}, the scale-invariant time dependent coefficient of the damping term for the transformed equation does not belong to the class of effective damping terms). Roughly speaking, assuming the above mentioned range for $\delta$, we guarantee $L^2-L^2$ estimates for \eqref{scale invariant wave equation with damp and mass} of \textquotedblleft parabolic type\textquotedblright $\,$ for the solution and its first order derivatives.

Let us report a brief historical overview, that is functional to elucidate our approach, on those papers in which this type of exponentially weighted Sobolev spaces is used in the study of semi-linear hyperbolic equations.

A first pioneering work in this direction is represented by 
\cite{TodYor01}, in which a global existence (in time) result in the space $\mathcal{C}([0,\infty),H^1)\cap \mathcal{C}^1([0,\infty),L^2)$ is proved for a classical damped wave equation with power non-linearity $|u|^p$, provided that the exponent  satisfies $p>p_{\Fuj}(n):=1+\frac{2}{n}$ and $p\leq p_{\GN}(n)$ for $n\geq 3$ and the data are compactly supported in $B_K(0)=\{x\in\mathbb{R}^n:|x|\leq K\}$. In particular, the weight used in the derivation of this result is $\mathrm{e}^{\psi_{0}(t,\cdot)}$, where
\begin{align*}
\psi_{0}(t,x)=\tfrac{1}{2}(t+K-\sqrt{(t+K)^2-|x|^2}) \quad \mbox{for} \,\, |x|<t+K.
\end{align*}  

Afterwards, 
in \cite{IkeTa05} the authors improved, for the same range of $p$, the previous result for the classical damped wave equation, removing the compactness assumption for the support of data and requiring instead the belonging of data to certain exponentially weighted spaces. This goal is achieved through a different choice of the weight function. Namely, instead of $\psi_0$ the authors consider
\begin{align*}
\psi_1(t,x)=\tfrac{|x|^2}{4(1+t)}.
\end{align*}

Subsequently, the damped wave equation with coefficient $b(t)=b_0(1+t)^{-\beta}$, where $b_0>0$, was studied for absorbing non-linearity (when $-1<\beta<1$) and for power and source non-linearity (whether $0\leq \beta <1$) in \cite{NishiZh09} and \cite{Nishi11}, respectively, for a suitable choice of the exponent function $\psi$. 

Indeed, in \cite{NishiZh09} some a-priori estimates are derived for solutions  to the equation 
\begin{equation*}
u_{tt}-\Delta u+b(t) u_t= -|u|^{p-1}u
\end{equation*} in the evolution space $\mathcal{C}([0,\infty),H^1)\cap \mathcal{C}^1([0,\infty),L^2)$, provided that $p$ satisfies $p>1$ and $p<\frac{n+2}{n-2}$ for $n\geq 3$ and the data are compactly supported.

 On the other hand, assuming $p>p_{\Fuj}(n)$ and $p<\frac{n+2}{n-2}$ if $n\geq 3$, 
 in \cite{Nishi11} a global existence result is proved for the equation
\begin{equation*}
u_{tt}-\Delta u+b(t) u_t= f(u),
\end{equation*}
where $f(u)=\pm |u|^p$ or $f(u)=|u|^{p-1}u$, in the case in which the data norm is sufficiently small in proper weighted Sobolev spaces. Furthermore, two blow-up results are proved in \cite{Nishi11} for these two different types of non-linearity.

Then, in \cite{DabbLucRei13} the authors, among the other things, generalize the global existence result in energy spaces with suitable exponential weight proved in \cite{Nishi11} to more general damping terms for $p>p_{\Fuj}(n)$ and $p\leq p_{\GN}(n)$ if $n\geq 3$. In fact, according to  
 the classification given in \cite{WiThe,Wi07}, in \cite{DabbLucRei13} the so-called \emph{effective damping term} $b(t)u_t$ is considered (to which the typology of damping  $b_0(1+t)^{-\beta}$ belongs for $-1<\beta<1$). 

Finally, in \cite{D13} the author proves a global existence result of small data solutions  for the semi-linear wave equation with scale-invariant damping term
\begin{equation*}
u_{tt}-\Delta u+\frac{\mu_1}{1+t} u_t= |u|^p,
\end{equation*}
assuming the condition $\mu_1\geq n+2$ for the coefficient of the damping term, $p>p_{\Fuj}(n)$ and $p\leq p_{\GN}(n)$ for $n\geq 3$ and the smallness of data in suitably chosen weighted Sobolev  spaces. More in detail, here the function 
\begin{align*}
\psi(t,x)=\frac{\mu_1 |x|^2}{2(1+t)^2}
\end{align*} is considered as exponent for the exponential weight. 
In our approach we will follow this choice of $\psi$. 

Indeed, we can slightly modify the estimate of \cite{DabbLucRei13,D13} in order to include
the additional scale-invariant mass term which is present in \eqref{CP semilinear 1}.

In the proof of the global existence (in time) result of small data solution a contradiction argument is used. For this reason it is necessary to prove at first a local (in time) existence result.

Finally, we report a blow-up result. Although this blow-up result provides the same range of $p$ in \eqref{CP semilinear 1}, that is derived in \cite{NunPalRei16} or that can be found by using the arguments of \cite{DabbLucMod}, for which the solution blows up in finite time under suitable conditions on the data, we will anyway provide a proof for this blow-up result. Indeed, the method we are going to write down is simpler and interesting in itself. Therefore, also for sake of completeness, we will include the proof of such blow-up result. Hence, combing the global existence result for small data solutions and the blow-up result, we determine explicitly the critical exponent for \eqref{CP semilinear 1} whether $\delta$ is sufficiently large.    
\subsection{Notations}

In this paper, we write $f \lesssim g$, when there exists a constant $C>0$ such that $f \leq Cg$. We write $f \approx g$ when $g \lesssim  f \lesssim g$.

 As we did in the introduction, we will use the notation $p_{\Fuj}(n)=1+\frac{2}{n}$ for the Fujita exponent.
 
Moreover, throughout the article we will denote by $\psi$ the function
\begin{equation}\label{definition of psi}
\psi(t,x)=\frac{\mu_1|x|^2}{2(1+t)^2}.
\end{equation}

For sake of brevity, we put 
\begin{align*}
b(t)=\frac{\mu_1}{1+t}, \quad m^2(t)=\frac{\mu_2^2}{(1+t)^2}
\end{align*} for the coefficients of damping and mass terms.

All function spaces consist of function defined on the whole space $\mathbb{R}^n$, therefore, we will not specify this fact in the notations.

Let $\sigma>0$ and $t\geq 0$. As in \cite{TodYor01} and \cite{IkeTa05}, we define the \emph{Sobolev spaces $L^2$ and $H^1$ with exponential weight $\mathrm{e}^{\sigma\psi(t,\cdot)}$ }
as follows:
\begin{align*}
L^2_{\sigma\psi(t,\cdot)}&:=\{f\in L^2:\|\mathrm{e}^{\sigma\psi(t,\cdot)}f\|_{L^2}<\infty\},\\
H^1_{\sigma\psi(t,\cdot)}&:=\{f\in H^1:\|\mathrm{e}^{\sigma\psi(t,\cdot)}f\|_{L^2}+\|\mathrm{e}^{\sigma\psi(t,\cdot)}\nabla f\|_{L^2}<\infty\},
\end{align*} with the norms 
 \begin{align*}
\|f\|_{L^2_{\sigma\psi(t,\cdot)}}&:=\|\mathrm{e}^{\sigma\psi(t,\cdot)}f\|_{L^2},\\
\|f\|_{H^1_{\sigma\psi(t,\cdot)}}&:=\|f\|_{L^2_{\sigma\psi(t,\cdot)}}+\|\nabla f\|_{L^2_{\sigma\psi(t,\cdot)}}.
\end{align*}

Finally, we denote by $\mathcal{A}$ the space
\begin{align}\label{space for initial data}
\mathcal{A}:=H^1_{\psi(0,\cdot)}\times L^2_{\psi(0,\cdot)}=H^1_{\frac{\mu_1}{2}|x|^2}\times L^2_{\frac{\mu_1}{2}|x|^2},
\end{align} that we will use for initial data.

\begin{rmk} Let us point out explicitly that the condition $(u_0,u_1)\in \mathcal{A}$ is stronger than the assumption $(u_0,u_1)\in (H^1\times L^2)\cap (L^1)^2$ made in \cite{NunPalRei16}. 

More generally,  for any $\sigma>0$ and $t\geq 0$ we have the embedding
\begin{align*}
L^2_{\sigma\psi(t,\cdot)}\hookrightarrow L^1\cap L^2.
\end{align*}  Indeed, by using Cauchy-Schwarz inequality and the positivity of $\psi$ we get
\begin{align}
\|f\|_{L^1}&\lesssim (1+t)^{\frac{n}{2}}\|f\|_{L^2_{\sigma\psi(t,\cdot)}},\label{L1-L2 weight est}\\
\|f\|_{L^2}&\leq \|f\|_{L^2_{\sigma\psi(t,\cdot)}},\label{L2-L2 weight est}
\end{align} where in the first inequality we use the value of the Gaussian integral. Then by H\"{o}lder's interpolation inequality we have also the embedding of $L^2_{\sigma\psi(t,\cdot)}$ in each $L^r$ for any $r\in [1,2]$.
\end{rmk}

\section{Main results}

Let us state the main theorems that will be proved in the present article.

\begin{thm}[Local existence]\label{thm local existence} Let $n\geq 1$ and $\mu_1>0$, $\mu_2^2$ be the nonnegative constants that appear in the damping and in the mass term in \eqref{CP semilinear 1}. 
Let us assume $p>1$ such that $p\leq \frac{n}{n-2}$ if $n\geq 3$.

 Then for each initial data $(u_0,u_1)\in \mathcal{A} 
 $ there exists a maximal existence time $T_m\in(0,\infty]$ such that the Cauchy problem \eqref{CP semilinear 1} has a unique solution $u\in\mathcal{C}([0,T_m),H^1)\cap \mathcal{C}^1([0,T_m),L^2)$ satisfying
\begin{align*}
\sup_{t\in[0,T]}\Big(\|u(t,\cdot)\|_{H^1_{\psi(t,\cdot)}}+\|u_t(t,\cdot)\|_{L^2_{\psi(t,\cdot)}}\Big)<\infty,
\end{align*} for any $T\in (0,T_m)$. 

Finally, if $T_m<\infty$, then
\begin{align*}
\limsup_{T\to T_m^-}\Big(\|u(t,\cdot)\|_{H^1_{\psi(t,\cdot)}}+\|u_t(t,\cdot)\|_{L^2_{\psi(t,\cdot)}}\Big)=\infty.
\end{align*}
\end{thm}
The previous local existence result is a prerequisite to obtain the next global existence result, whose proof is based on a contradiction argument that requires the existence of local in time solutions for \eqref{CP semilinear 1}.  For the proof of Theorem \ref{thm local existence} we will follow the approach of \cite[Appendix A]{IkeTa05}.

\begin{thm}[Global existence of small data solutions]\label{thm glob exi exp data} Let $n\geq 1$ and $\mu_1>0$, $\mu_2$ be nonnegative constants such that $\delta \geq (n+1)^2$. 
Let us consider 
\begin{align*}
p>p_{\Fuj}\left(n+\tfrac{\mu_1-1}{2}-\tfrac{\sqrt{\delta}}{2}\right) \quad\mbox{such that} \,\,\, p\leq \tfrac{n}{n-2} \,\,\, \mbox{if} \,\,\, n\geq 3.
\end{align*}

 Then there exists $\varepsilon_0>0$ such that for any initial data 
\begin{align}\label{global existence eponential weight data cond}
(u_0,u_1)\in \mathcal{A}\quad \mbox{satisfying}\quad \|(u_0,u_1)\|_\mathcal{A}\leq \varepsilon_0
\end{align} there is a unique solution $u\in \mathcal{C}\big([0,\infty),H^1_{\psi(t,\cdot)}\big)\cap \mathcal{C}^1\big([0,\infty),L^2_{\psi(t,\cdot)}\big)$ to the Cauchy problem \eqref{CP semilinear 1}.
Moreover, $u$ satisfies the following estimates:
\begin{align*}
\|u(t,\cdot)\|_{L^2}&\lesssim (1+t)^{-\frac{n}{2}-\frac{\mu_1}{2}+\frac{1}{2}+\frac{\sqrt{\delta}}{2}}\|(u_0,u_1)\|_\mathcal{A},\\
\|(\nabla u,u_t)(t,\cdot)\|_{L^2\vphantom{L^2_{\psi(t,\cdot)}}}&\lesssim (1+t)^{-\frac{n}{2}-\frac{\mu_1}{2}-\frac{1}{2}+\frac{\sqrt{\delta}}{2}}\ell_\delta(t)\|(u_0,u_1)\|_\mathcal{A},\\
\|u(t,\cdot)\|_{L^2_{\psi(t,\cdot)}}&\lesssim (1+t)\|(u_0,u_1)\|_\mathcal{A},\\
\|(\nabla u,u_t)(t,\cdot)\|_{L^2_{\psi(t,\cdot)}}&\lesssim \|(u_0,u_1)\|_\mathcal{A},\\
\end{align*} with $$\ell_\delta(t)=\begin{cases}1 &\mbox{if}\quad \delta>(n+1)^2, \\1+\left(\log(1+t)\right)^\frac{1}{2} &\mbox{if}\quad \delta=(n+1)^2. \end{cases}$$
\end{thm}

\begin{rmk} Let us underline that the decay rates in the unweighted energy estimates are the same found in \cite[Theorem 2.2]{NunPalRei16}. Nevertheless, thanks to the approach with weighted estimates, we can avoid the restriction $p\geq 2$ and, consequently, we get a result for every space dimension $n\geq 1$ and for the same range of allowed values of $\delta$ and $p$.
\end{rmk}

\begin{thm}[Blow-up]\label{thm blow up} Let $n\geq 1$ and $\mu_1,\mu_2^2$ be nonnegative constants such that $\delta\geq 0$. Let us assume that $u\in \mathcal{C}^2([0,T)\times \mathbb{R}^n)$ is a solution to \eqref{CP semilinear 1}, where $T>0$ is the life-span of $u$, with compactly supported initial data such that
\begin{align}
&\int_{\mathbb{R}^n}u_0(x)dx>0\, ,\qquad \int_{\mathbb{R}^n}\left(u_1(x)+\left(\tfrac{\mu_1-1}{2}-\tfrac{\sqrt{\delta}}{2}\right)u_0(x)\right)dx>0\, . \label{positivity assumption data}
\end{align} 

If the exponent $p$ of the non-linearity satisfies 
\begin{align}
1<p\leq p_{\Fuj}\left(n+\tfrac{\mu_1-1}{2}-\tfrac{\sqrt{\delta}}{2}\right),\label{upper bound p blow-up}
\end{align}  then $T<\infty$, that is the solutions $u$ blows up in finite time.
\end{thm}

\begin{rmk} Comparing the assumptions on $p$ in Theorem \ref{thm blow up} with those of \cite[Theorem 2.4]{NunPalRei16}, we see that solutions to \eqref{CP semilinear 1} blow up in finite times for the same range of $p$, for example in the case in which nontrivial data satisfy
\begin{align*}
u_0\geq 0, \qquad u_1(x)+\left(\tfrac{\mu_1-1}{2}-\tfrac{\sqrt{\delta}}{2}\right)u_0(x) \geq 0.
\end{align*} 

Although, in Theorem \ref{thm blow up} we require as additional condition the compactness of the support of data, the proof itself is interesting. Indeed, this proof is based on a blow-up dynamic for an ordinary differential inequality with polynomial non-linearity. Philosophically, this result for ordinary differential inequalities plays the role that \emph{Kato's Lemma} has in the proof of the blow-up of solutions to the free wave equation for exponents below the Strauss exponent and suitable data (for Kato's Lemma see for example \cite{Zhou03,Yag05,Yor06,DabbLucRei15}). In particular, our approach follows results and ideas used firstly in \cite{TodYor01} and then in \cite{Zhou05,Nishi11}. 
Moreover, thanks to this ordinary differential inequality, we can obtain also an upper-bound for the life-span of the solution.

Dissimilarly, in \cite[Theorem 2.4]{NunPalRei16} the so-called \emph{method of test function} is applied (for further references to this method cf. \cite{Zhang,Waka14}).
\end{rmk}

\section{Overview on our approach}

We plan to apply Duhamel's principle to write the solution to \eqref{CP semilinear 1}. Because the linear equation related to the semi-linear equation in \eqref{CP semilinear 1} is not invariant by time translations, we have to derive estimates for the family of linear parameter dependent Cauchy problems

\begin{align}\label{linear parameter dependent CP}
\begin{cases}
u_{tt}-\Delta u+\frac{\mu_1}{1+t} u_t+\frac{\mu_2^2}{(1+t)^2}u=0, & t>s\geq 0, \,\, x\in \mathbb{R}^n, \\ u(s,x)=u_0(x), & x\in\mathbb{R}^n, \\ u_t(s,x)=u_1(x), & x\in\mathbb{R}^n.
\end{cases}
\end{align} 
Since we use Duhamel's principle in the study of the non-linear problem, the case in which $u_0(x)\equiv 0$ is particularly important. 

We fix now some notations for our problem.
Let us denote by $E_0(t,s,x)$, $ E_1(t,s,x)$ the fundamental solutions to the Cauchy problem \eqref{linear parameter dependent CP}, i.e. the distributional solutions with data $(u_0,u_1)=(\delta_0,0)$ and $(u_0,u_1)=(0,\delta_0)$, respectively, taken at the time $s\geq 0$, where $\delta_0$ is the Dirac distribution in the $x$ variable. Therefore, if $\ast_{(x)}$ denotes the convolution with respect to the $x$ variable, by using the linearity of the equation in \eqref{linear parameter dependent CP}, it is possible to represent the solution to the Cauchy problem \eqref{linear parameter dependent CP} as 
\begin{align*}
u(t,x)=E_0(t,s,x)\ast_{(x)}u_0(x)+E_1(t,s,x)\ast_{(x)} u_1(x).
\end{align*}

Now we clarify the type of solutions to \eqref{CP semilinear 1} we are interested in. 

According to Duhamel's principle, we get 
\begin{align*}
u(t,x)=\int_0^t E_1(t,s,x)\ast_{(x)} F(s,x) ds
\end{align*} as solution to the inhomogeneous Cauchy problem 
\begin{align*}
\begin{cases}
u_{tt}-\Delta u+\frac{\mu_1}{1+t} u_t+\frac{\mu_2^2}{(1+t)^2}u=F(t,x), & t> 0, \,\, x\in \mathbb{R}^n, \\ u(0,x)=0, & x\in\mathbb{R}^n, \\ u_t(0,x)=0, & x\in\mathbb{R}^n.
\end{cases}
\end{align*} 

Hence, we consider as solutions to \eqref{CP semilinear 1} on $(0,T)\times \mathbb{R}^n$ any fixed point of the operator $N$ defined as follows:
\begin{align}
u\in X(T)\to Nu(t,x) &:=E_0(t,s,x)\ast_{(x)}u_0(x)+E_1(t,s,x)\ast_{(x)} u_1(x) \notag \\&\quad +\int_0^t E_1(t,s,x)\ast_{(x)} |u(s,x)|^p ds \label{definition N operator}
\end{align} for a properly chosen space $X(T)$, where $T$ denotes the life span of the solution.

Consequently both local or global (in time) existence results are based on the following type of inequalities:
\begin{align}
\| Nu \|_{X(T)}&\leq C_0(T) + C_1(T) \|u\|_{X(T)}^p , \label{contraction prociple 1st ineq}\\
\| Nu -Nv \|_{X(T)}&\leq  C_2(T) \|u-v\|_{X(T)}\big(\|u\|_{X(T)}^{p-1}+\|v\|_{X(T)}^{p-1} \big), \label{contraction prociple 2nd ineq}
\end{align}
where $X(T)$ is a suitable Banach space and $C_0$ depends on the norm of initial data.

Indeed, if  $C_0(T)$ is bounded as $T\to 0$ and $C_1(T),C_2(T)\to 0$ as $T\to 0$ , then we obtain from \eqref{contraction prociple 1st ineq} and \eqref{contraction prociple 2nd ineq} a local (in time) existence result for large data thanks to Banach's fixed point theorem.

Similarly, whether $C_0(T)$ is constant and $C_1(T),C_2(T)$ are bounded as $T\to \infty$, it follows a global (in time) existence result for small data.

 Nevertheless, for Theorem \ref{thm glob exi exp data}, as we announced in the introduction, we do not follow this standard approach. Roughly speaking, after proving Theorem \ref{thm local existence}, we will suppose by contradiction that our local in time solution can not be prolonged for all times, regardless of the smallness of initial data. Then, considering a suitable norm on $X(T)$, we show the uniform boundedness of the weighted energy (cf. Section \ref{LE Section}), provided the smallness of data. However this contradicts the assumption of not unlimited prolongability for any times of the solution.
 
This means, in particular, that in Theorem \ref{thm glob exi exp data} the global in time solutions we are interested in are solutions to the integral equation 
 \begin{align*}
 u(t,x)&=E_0(t,s,x)\ast_{(x)}u_0(x)+E_1(t,s,x)\ast_{(x)} u_1(x)\\& \quad +\int_0^t E_1(t,s,x)\ast_{(x)} |u(s,x)|^p ds
 \end{align*} which can be extended for all positive times.

From the previous considerations it follows that the difficulties in the proof of a local or global existence result for large or small data, respectively, is reduced to the choice of the space $X(T)$ and to the verification of \eqref{contraction prociple 1st ineq} and \eqref{contraction prociple 2nd ineq}. 

In this paper, we restrict our consideration to the space
\begin{align*}
X(T)=\mathcal{C}\big([0,T],H^1_{\psi(t,\cdot)}\big)\cap \mathcal{C}^1\big([0,T],L^2_{\psi(t,\cdot)}\big),
\end{align*} both in Theorem \ref{thm local existence} and in Theorem \ref{thm glob exi exp data}. As we will see, the crucial difference lies in the choice of the norm for $X(T)$ (cf. Section \ref{LE Section} and Section \ref{SDGE Section}).

Let us spend now few words on the function $\psi$, defined in \eqref{definition of psi}, and its useful properties which will be helpful in the proof of our main results. 

This function satisfies the following relations:
\begin{align}
|\nabla \psi(t,x)|^2+b(t)\psi_t(t,x)&=0, \label{prop grad psi}\\
\Delta\psi(t,x)&=\frac{n\mu_1}{(1+t)^2}>0\label{prop lapl psi}
\end{align}
for any $t\geq 0$ and $x\in\mathbb{R}^n$. 

The equation \eqref{prop grad psi} is related to the symbol of the linear parabolic equation $b(t)u_t-\Delta u=0$, that is, we have in mind the parabolic effect when we consider the weight $\mathrm{e}^{\psi(t,x)}$. 

 In Sections \ref{LE Section} and \ref{SDGE Section} we will employ several times the following two fundamental relations.
 
  The first one is the following equality:
\begin{align}
&\mathrm{e}^{2\psi}u_t\left(u_{tt}-\Delta u +b(t)u_t+m^2(t)u\right) =\frac{\partial}{\partial t}\left(\frac{\mathrm{e}^{2\psi}}{2}\left(u_t^2+|\nabla u|^2+m^2(t)u^2\right)\right)\notag\\ &\quad -\diver(\mathrm{e}^{2\psi}u_t\nabla u)+\frac{\mathrm{e}^{2\psi}}{\psi_t}u_t^2\left(|\nabla\psi|^2+b(t)\psi_t\right)-\frac{\mathrm{e}^{2\psi}}{\psi_t}|u_t\nabla\psi-\psi_t\nabla u|^2\notag\\ & \quad -\psi_t \mathrm{e}^{2\psi}(u_t^2+m^2(t)u^2)-\frac{1}{2}\mathrm{e}^{2\psi}u^2\frac{d}{dt}m^2(t).\label{fundamental equality 1}
\end{align} In order to verify \eqref{fundamental equality 1} one can make use of the following relations:
\begin{align*}
\mathrm{e}^{\psi}u_t u_{tt}& 
=\frac{\partial}{\partial t}\left(\frac{\mathrm{e}^{2\psi}}{2}u_t^2\right)-\psi_t \mathrm{e}^{2\psi}u_t^2;\\
\mathrm{e}^{\psi}u_t \Delta u &
=\diver(\mathrm{e}^{2\psi}u_t\nabla u)-2u_t \mathrm{e}^{2\psi}\nabla\psi\cdot \nabla u-\frac{\partial}{\partial t}\left(\frac{\mathrm{e}^{2\psi}}{2}|\nabla u|^2\right)\\& \quad +\psi_t \mathrm{e}^{2\psi}|\nabla u|^2;\\
2\mathrm{e}^{2\psi}u_t\nabla \psi \cdot \nabla u &
=\frac{\mathrm{e}^{2\psi}}{\psi_t}\left(u_t^2|\nabla\psi|^2+\psi_t^2|\nabla u|^2-|u_t\nabla\psi-\psi_t\nabla u|^2\right);\\
\mathrm{e}^{\psi}u_t m^2(t)u& 
=\frac{\partial}{\partial t}\left(\frac{\mathrm{e}^{2\psi}}{2}m^2(t)u^2\right)-\psi_t \mathrm{e}^{2\psi}m^2(t)u^2-\frac{\mathrm{e}^{2\psi}}{2}u^2\frac{d}{dt}m^2(t).
\end{align*}
The second one is the upcoming inequality. If $u$ is solution of the equation then \eqref{CP semilinear 1}, since
\begin{align*}
\mathrm{e}^{2\psi}u_t|u|^p
=\frac{\partial}{\partial t}\left(\mathrm{e}^{2\psi}\frac{|u|^pu}{p+1}\right)-2\psi_t \mathrm{e}^{2\psi}\frac{|u|^pu}{p+1}
\end{align*} from \eqref{prop grad psi} and \eqref{fundamental equality 1} we get immediately
\begin{align}
\frac{\partial}{\partial t}\left(\frac{\mathrm{e}^{2\psi}}{2}\left(u_t^2+|\nabla u|^2+m^2(t)u^2\right)-\mathrm{e}^{2\psi}\frac{|u|^pu}{p+1}\right)
 &\leq \diver(\mathrm{e}^{2\psi}u_t\nabla u)\notag\\ &\quad-2\psi_t \mathrm{e}^{2\psi}\frac{|u|^pu}{p+1},\label{fundamental equality 2}
\end{align} where we used that $\psi_t\leq 0$ and the fact that $m^2(t)$ is a strictly decreasing function.

 In Sections \ref{LE Section} and \ref{SDGE Section} a fundamental role in the derivation of energy estimates will be played by \eqref{fundamental equality 1} and \eqref{fundamental equality 2}.

\section{Local existence: proof of Theorem \ref{thm local existence}}\label{LE Section}

In the proof of Theorems \ref{thm local existence} and \ref{thm glob exi exp data} we make use of the following inequalities. Although these are slight modifications of well known inequalities proved in \cite{TodYor01, IkeTa05}, for sake of self-completeness we include also their proofs.

\begin{lemma}\label{1 lemma GN with weight} Let $\sigma>0$, $t\geq 0$ and $v\in H^1_{\sigma\psi(t,\cdot)}$. Then it holds 
\begin{align*}
\sigma\mu_1 n(1+t)^{-2}\|\mathrm{e}^{\sigma\psi(t,\cdot)}v\|^2_{L^2}+\|\nabla(\mathrm{e}^{\sigma\psi(t,\cdot)}v)\|^2_{L^2}\leq \|\mathrm{e}^{\sigma\psi(t,\cdot)}\nabla v\|^2_{L^2}.
\end{align*}
\end{lemma}

\begin{proof}
We put $f=\mathrm{e}^{\sigma\psi}v$. Then
\begin{align*}
\nabla v=\nabla(\mathrm{e}^{-\sigma\psi}f)=-\sigma \mathrm{e}^{-\sigma\psi}f\nabla\psi+\mathrm{e}^{-\sigma\psi}\nabla f ,\quad  \mathrm{e}^{\sigma\psi}\nabla v=\nabla f-\sigma f\nabla \psi,
\end{align*}
respectively.

 Hence, 
\begin{align*}
\|\mathrm{e}^{\sigma\psi(t,\cdot)}\nabla v\|^2_{L^2}&=\|\nabla f(t,\cdot)\|^2_{L^2}\!+\!\sigma^2 \|(f\nabla \psi)(t,\cdot)\|^2_{L^2}\\& \qquad -2\sigma\big(\nabla f(t,\cdot), (f\nabla\psi)(t,\cdot)\big)_{L^2}. 
\end{align*}
 Therefore, integrating by parts we get
\begin{align*}
\big(\nabla f(t,\cdot), (f\nabla\psi)(t,\cdot)\big)_{L^2}&
=-\frac{1}{2}\int_{\mathbb{R}^n}f^2(t,x) \Delta\psi(t,x) dx.
\end{align*} 

Consequently, from \eqref{prop lapl psi} we obtain
\begin{align*}
\|\mathrm{e}^{\sigma\psi(t,\cdot)}\nabla v\|^2_{L^2}& 
\geq \|\nabla f(t,\cdot)\|^2_{L^2}+\sigma\mu_1 n(1+t)^{-2}\|f(t,\cdot)\|^2_{L^2},
\end{align*} which is exactly the searched estimate. 
\end{proof}

\begin{lemma}\label{2 lemma GN with weight} Let $\theta(q)=n\big(\frac{1}{2}-\frac{1}{q}\big)$ with $\theta\in [0,1]$ and let $\sigma\in (0,1]$, $t\geq 0$.\\ If $v\in H^1_{\psi(t,\cdot)}$, then it holds the inequality 
\begin{align*}
\|\mathrm{e}^{\sigma\psi(t,\cdot)}v\|_{L^q}\leq C (1+t)^{1-\theta(q)}\|\nabla v\|_{L^2}^{1-\sigma}\|\mathrm{e}^{\psi(t,\cdot)}\nabla v\|_{L^2}^{\sigma},
\end{align*} for a nonnegative constant $C$  that does not depend on $v$ and $t$.
\end{lemma}

\begin{proof}
Let us prove preliminary that $v\in H^1_{\sigma\psi(t,\cdot)}$ for any $\sigma\in (0,1]$.

 By H\"{o}lder's inequality we get
\begin{align}
\|\mathrm{e}^{\sigma\psi(t,\cdot)}\nabla v\|_{L^2}^2&
\leq \|\mathrm{e}^{\psi(t,\cdot)}\nabla v\|^{2\sigma}_{L^2}\|\nabla v\|^{2(1-\sigma)}_{L^2}.\label{2 lemma GN with weight 1 est}
\end{align}
 In the same way 
\begin{align*}
\|\mathrm{e}^{\sigma\psi(t,\cdot)}v\|_{L^2}^2&
\leq \|\mathrm{e}^{\psi(t,\cdot)} v\|^{2\sigma}_{L^2}\| v\|^{2(1-\sigma)}_{L^2}.
\end{align*}

From Lemma \ref{1 lemma GN with weight} we get for the function $f=\mathrm{e}^{\sigma\psi}v$ that $f(t,\cdot)\in H^1$ and
\begin{align}
\|f(t,\cdot)\|_{L^2}&\lesssim (1+t)\|\mathrm{e}^{\sigma\psi(t,\cdot)}\nabla v\|_{L^2},\label{2 lemma GN with weight 2 est}\\ \|\nabla f(t,\cdot)\|_{L^2}&\leq \|\mathrm{e}^{\sigma\psi(t,\cdot)}\nabla v\|_{L^2} \label{2 lemma GN with weight 3 est},
\end{align}
for any $t\geq 0$. 

By the classical Gagliardo-Nirenberg inequality we have 
\begin{align*}
\|f(t,\cdot)\|_{L^q}\lesssim \|f(t,\cdot)\|_{L^2}^{1-\theta(q)}\|\nabla f(t,\cdot)\|_{L^2}^{\theta(q)},
\end{align*} where $\theta(q)=n\big(\frac{1}{2}-\frac{1}{q}\big)$ (see \cite{Fried76}). 

Thus, combining \eqref{2 lemma GN with weight 1 est}, \eqref{2 lemma GN with weight 2 est} and \eqref{2 lemma GN with weight 3 est} with the previous estimate we have
\begin{align*}
\|f(t,\cdot)\|_{L^q}&\lesssim (1+t)^{1-\theta(q)}\|\mathrm{e}^{\sigma\psi(t,\cdot)}\nabla v\|_{L^2} \\ &\leq (1+t)^{1-\theta(q)}\|\mathrm{e}^{\psi(t,\cdot)}\nabla v\|^{\sigma}_{L^2}\|\nabla v\|^{1-\sigma}_{L^2}.
\end{align*} This completes the proof.
\end{proof}

In the proof of Theorem \ref{thm local existence} we will use also the next result, which is a generalization to the non-linear case of Gronwall's lemma (for the proof it is possible to see, for example, \cite{Cord71}).
\begin{lemma}[Bihari's inequality]\label{lemma of gronwall type} Let $k$ be a nonnegative, continuous function, $M$ a real constant and $g$ a continuous, non-decreasing, nonnegative function such that $$G(u)=\int_0^u \frac{ds}{g(s)}$$ is well defined.
Let $y$ be a continuous function such that 
\begin{equation*}y(t)\leq M+\int_0^t k(s) g(y(s))ds \qquad \mbox{ for any} \,\, t\geq 0.
\end{equation*} Then 
\begin{equation*}
G(y(t))\leq G(M)+\int_0^t k(s)ds\qquad \mbox{ for any} \,\, t\geq 0.
\end{equation*}
\end{lemma}

Finally, before starting with the proof of Theorem \ref{thm local existence}, we prove that the space we will work with is actually a Banach space and then, consequently, we can apply Banach's fixed point theorem.

\begin{lemma}\label{lemma E(T) banach} Let us consider the space 
\begin{align*}
\mathrm{E}(T):=\big\{u\in\mathcal{C}([0,T],H^1)\,\cap\, \mathcal{C}^1([0,T],L^2): \|u\|_T^\psi<\infty\big\},
\end{align*}
 where 
\begin{align*}
\|u\|_T^\psi:=\sup_{t\in[0,T]}\left(\|u(t,\cdot)\|_{H^1_{\psi(t,\cdot)}}+\|u_t(t,\cdot)\|_{L^2_{\psi(t,\cdot)}}\right).
\end{align*} 

Then $\big(\mathrm{E}(T), \|\cdot\|_T^\psi\big)$ is a Banach space.
\end{lemma}

\begin{proof}
Let $\{u_k\}_{k\geq 1}\subset \mathrm{E}(T)$ be a Cauchy sequence with respect to the norm $\|\cdot\|_T^\psi$.
 According to the definition of the norm in $\mathrm{E}(T)$, we have that $\{\mathrm{e}^\psi u_k\}_{k\geq 1}$, $\{\mathrm{e}^\psi \partial_t u_k\}_{k\geq 1}$ and $\{\mathrm{e}^\psi \partial_{xj}u_k\}_{k\geq 1}$ for any $j=1,\dots ,n$ are Cauchy sequences in $\mathcal{C}([0,T],L^2)$ which is a Banach space. 
 
 So we can find $\tilde{u},\tilde{u}_0,\tilde{u}_1,\dots,\tilde{u}_n\in \mathcal{C}([0,T],L^2)$ such that we have the following convergences in $\mathcal{C}([0,T],L^2)$ as $ k\to\infty$:
 \begin{align}
 \mathrm{e}^\psi u_k &\to  \tilde{u}, \label{convergence function}\\ 
 \mathrm{e}^\psi \partial_t u_k &\to  \tilde{u}_0 ,\label{convergence time deriv}\\
 \mathrm{e}^\psi \partial_{x_j} u_k &\to \tilde{u}_j  \quad \mbox{for any }\,\, j=1,\dots ,n. \label{convergence grad}
 \end{align}  Let us define $u=\mathrm{e}^{-\psi}\tilde{u}$ and $\bar{u}_j=\mathrm{e}^{-\psi}\tilde{u}_j$ for any $j=0,1,\dots ,n$.  It is clear that $u,\bar{u}_0,\bar{u}_1,\dots,\bar{u}_n\in \mathcal{C}\big([0,T],L^2_{\psi(t,\cdot)}\big)$.
 
  Therefore, if we prove that $\partial_t u=\bar{u}_0$ and $\partial_{x_j}u=\bar{u}_j$ for any $j=1,\dots,n$, by \eqref{convergence function}, \eqref{convergence time deriv} and \eqref{convergence grad} it follows immediately that $u$ belongs to the space $ \mathcal{C}\big([0,T],H^1_{\psi(t,\cdot)}\big)\cap\mathcal{C}^1\big([0,T],L^2_{\psi(t,\cdot)}\big)$ and that $u_k\to u$ in $\mathrm{E}(T)$ as $k\to\infty$.
  
For all test functions $\phi\in\mathcal{C}^1_0$
and indexes $j=1,\dots,n$, using Cauchy-Schwarz inequality, \eqref{convergence function} and \eqref{convergence grad} we obtain
\begin{align*}
\int_{\mathbb{R}^n}\partial_{x_j}\phi(x)& u(t,x)dx=\int_{\mathbb{R}^n}\partial_{x_j}\phi(x)\mathrm{e}^{-\psi(t,x)}\tilde{u}(t,x)dx\\&=\lim_{k\to\infty}\int_{\mathbb{R}^n}\partial_{x_j}\phi(x)u_k(t,x)dx=-\lim_{k\to\infty}\int_{\mathbb{R}^n}\phi(x)\partial_{x_j}u_k(t,x)dx\\&=-\int_{\mathbb{R}^n}\phi(x)\mathrm{e}^{-\psi(t,x)}\tilde{u}_j(t,x)dx=-\int_{\mathbb{R}^n}\phi(x)\bar{u}_j(t,x)dx,
\end{align*} that is, $\partial_{x_j}u=\bar{u}_j$ in the Sobolev sense.

Due to the fundamental theorem of calculus for vector-valued functions, we have 
\begin{align*}
u_k(t,x)=u_k(0,x)+\int_0^t \partial_t u_k(s,x)ds \quad \mbox{in} \quad L^2
\end{align*}
for any $k\geq 1$.

  Taking the limit as $k\to\infty$ we get
\begin{align*}
u(t,x)&=\mathrm{e}^{-\psi(t,x)}\tilde{u}(t,x)=\mathrm{e}^{-\psi(0,x)}\tilde{u}(0,x)+\int_0^t \mathrm{e}^{-\psi(s,x)}\tilde{u}_0(s,x)ds\\&=u(0,x)+\int_0^t \bar{u}_0(s,x)ds \quad \mbox{in} \quad L^2.
\end{align*}
 Thus using again the fundamental theorem of calculus we have  $\bar{u}_0=\partial_t u$. 
\end{proof}

\begin{proof}[Proof of Theorem \ref{thm local existence}]
Let $T,K>0$. We define 
\begin{align*}
B_{T,K}^\psi=\{v\in \mathcal{C}([0,T],H^1)\cap\mathcal{C}^1([0,T],L^2):\|v\|_T^\psi\leq K\},
\end{align*} where $\|\cdot\|_T^\psi$ denotes the same norm introduced in the statement of Lemma \ref{lemma E(T) banach}.

Let us consider the map 
\begin{align*}
\Phi: B_{T,K}^\psi&\longrightarrow \mathcal{C}([0,T],H^1)\cap\mathcal{C}^1([0,T],L^2), \\  v&\longmapsto u=\Phi(v),
\end{align*} where $u$ is the unique solution to the Cauchy problem
\begin{align*}
\begin{cases}
u_{tt} -\Delta u + \frac{\mu_1}{1+t} u_t + \frac{\mu_2^2}{(1+t)^2} u
=|v|^p,\quad (t,x)\in (0,T)\times\mathbb{R}^n, \\  u(0,x)=u_0(x), \quad x\in \mathbb{R}^n,\\  u_t(0,x)=u_1(x),\quad x\in \mathbb{R}^n.
\end{cases}
\end{align*} Our goal is to prove that, for a suitable choice of $T$ and $K$, $\Phi$ is a contraction map from $B_{T,K}^\psi$ into itself. From the relation \eqref{fundamental equality 1} it follows that
\begin{align*}
\mathrm{e}^{2\psi}u_t|v|^p\geq \frac{\partial}{\partial t}\bigg(\frac{\mathrm{e}^{2\psi}}{2}(u_t^2+|\nabla u|^2+m^2(t)u^2)\bigg)-\diver(\mathrm{e}^{2\psi}u_t\nabla u).
\end{align*}

 Therefore, denoting by
\begin{align*}
E_{\psi,u}(t)=\frac{1}{2}\int_{\mathbb{R}^n}\mathrm{e}^{2\psi(t,x)}\left(|u_t(t,x)|^2+|\nabla u(t,x)|^2+m^2(t)|u(t,x)|^2\right)dx
\end{align*} the \emph{weighted energy} of the function $u$, integrating over $[0,t]\times\mathbb{R}^n$ the previous inequality, using the divergence theorem in a weak sense for $L^1$ functions, one has
\begin{align*}
E_{\psi,u}(t)\leq E_{\psi,u}(0)+\int_0^t\int_{\mathbb{R}^n}\mathrm{e}^{2\psi(s,x)}u_t(s,x)|v(s,x)|^p dx ds.
\end{align*} 

Applying Cauchy-Schwarz inequality we get
\begin{align*}
E_{\psi,u}(t)
&\leq E_{\psi,u}(0)+\sqrt{2}\int_0^t\left(\int_{\mathbb{R}^n}\mathrm{e}^{2\psi(s,x)}|v(s,x)|^{2p} dx\right)^{\frac{1}{2}} E_{\psi,u}(s)^{\frac{1}{2}} ds.
\end{align*}
Because of Bihari's inequality, with $g(u)=(2u)^{\frac{1}{2}}$, we find
\begin{align}\label{Bihari est}
E_{\psi,u}(t)^{\frac{1}{2}}\leq E_{\psi,u}(0)^{\frac{1}{2}}+\frac{1}{\sqrt{2}}\int_0^t\left(\int_{\mathbb{R}^n}\mathrm{e}^{2\psi(s,x)}|v(s,x)|^{2p} dx\right)^{\frac{1}{2}} ds.
\end{align} 
Since $v\in B_{T,K}^\psi$, for any $t\in [0,T]$ it results $v(t,\cdot)\in H^1_{\psi(t,\cdot)}$. Hence, from Lemma \ref{2 lemma GN with weight} we get 
\begin{align*}
\|\mathrm{e}^{\frac{1}{p}\psi(s,\cdot)}v(s,\cdot)\|^{2p}_{L^{2p}}
& \lesssim (1+s)^{2p(1-\theta(2p))}\|\nabla v(s,\cdot)\|_{L^2}^{2(p-1)}\|\mathrm{e}^{\psi(s,\cdot)}\nabla v(s,\cdot)\|_{L^2}^{2}\\&\lesssim 
(1+s)^{2p(1-\theta(2p))}K^{2p}.
\end{align*} 

Consequently, from \eqref{Bihari est} we obtain
\begin{align*}
E_{\psi,u}(t)^{\frac{1}{2}}\leq E_{\psi,u}(0)^{\frac{1}{2}}+T(1+T)^{p(1-\theta(2p))}K^p.
\end{align*}
Then on one hand we have
\begin{align*}
\|\mathrm{e}^{\psi(t,\cdot)}u_t(t,\cdot)\|_{L^2}+\|\mathrm{e}^{\psi(t,\cdot)}\nabla u(t,\cdot)\|_{L^2}\lesssim E_{\psi,u}(0)^{\frac{1}{2}}+T(1+T)^{p(1-\theta(2p))}K^p,
\end{align*} on the other hand 
\begin{align*}
\|\mathrm{e}^{\psi(t,\cdot)}u(t,\cdot)\|_{L^2}&\lesssim m^{-1}(t)E_{\psi,u}(t)^{\frac{1}{2}}\\&\lesssim (1+T)E_{\psi,u}(0)^{\frac{1}{2}}+T(1+T)^{p(1-\theta(2p))+1}K^p.
\end{align*}

Summarizing, for all $t\in[0,T]$
\begin{align*}
\|u(t,\cdot)\|_{H^1_{\psi(t,\cdot)}}\!+\|u_t(t,\cdot)\|_{L^2_{\psi(t,\cdot)}}\!\lesssim (1+T)E_{\psi,u}(0)^{\frac{1}{2}}\!+T(1+T)^{p(1-\theta(2p))+1}K^p.
\end{align*}
Since the initial energy depends only on the data, we can choose $K$ sufficiently large such that the first term in the above inequality is less than $\tfrac{K}{2}$, while fixing $T>0$ enough small also the second term can be estimated with $\tfrac{K}{2}$. 

Being the above estimate uniform in $t$, it follows that $\|v\|_{T}^\psi\leq K$, that is $\Phi$ maps $B_{T,K}^\psi$ to itself. 

Now we have to prove that $\Phi$ is a contraction map, provided that $T$ is sufficiently small. Let us consider $v,\bar{v}\in B_{T,K}^\psi$. Denoting $u=\Phi(u), \bar{u}=\Phi(\bar{v})$, it follows immediately that $w=u-\bar{u}$ satisfies the Cauchy problem
\begin{equation*}\begin{cases}
 w_{tt} -\Delta w + \frac{\mu_1}{1+t} w_t + \frac{\mu_2^2}{(1+t)^2} w
=|v|^p-|\bar{v}|^p,\quad (t,x)\in (0,T)\times\mathbb{R}^n, \\  u(0,x)=u_t(0,x)=0, \quad x\in \mathbb{R}^n.
\end{cases}
\end{equation*} 

Then using once again \eqref{fundamental equality 1} and the divergence theorem, we get after integrating over $[0,t]\times \mathbb{R}^n$ the inequality
\begin{align*}
E_{\psi,w}(t)\leq \int_0^t\int_{\mathbb{R}^n}\mathrm{e}^{2\psi(s,x)}\Big(|v(s,x)|^p-|\bar{v}(s,x)|^p\Big)w_t(s,x)dx ds.
\end{align*} Using the inequality $||v|^p-|\bar{v}|^p|\leq p|v-\bar{v}|(|v|+|\bar{v}|)^{p-1}$ and Cauchy-Schwarz inequality we find
\begin{align*}
E_{\psi,w}(t)&\lesssim  \int_0^t\!\! \int_{\mathbb{R}^n}\!\!\mathrm{e}^{2\psi(s,x)}|v(s,x)-\bar{v}(s,x)|\Big(|v(s,x)|+|\bar{v}(s,x)|\Big)^{p-1}\!\!\!w_t(s,x)dx ds \\
&\leq \int_0^tE_{\psi,w}(s)^{\frac{1}{2}}\bigg(\int_{\mathbb{R}^n}\!\mathrm{e}^{2\psi(s,x)}|v(s,x)-\bar{v}(s,x)|^2\\ &\qquad \qquad\qquad\qquad\times\Big(|v(s,x)|+|\bar{v}(s,x)|\Big)^{2(p-1)}dx\bigg)^{\frac{1}{2}} ds.
\end{align*}
Applying once again Lemma \ref{lemma of gronwall type}, we get the inequality
\begin{align}\label{stima intermedia lemma loc esistenza}
E_{\psi,w}(t)^\frac{1}{2}\lesssim  \! \int_0^t \!\left(\int_{\mathbb{R}^n}\!\!\!\mathrm{e}^{2\psi(s,x)}|v(s,x)\!-\!\bar{v}(s,x)|^2(|v(s,x)|+|\bar{v}(s,x)|)^{2(p-1)}\!dx\!\right)^{\frac{1}{2}} \!\!\!ds.
\end{align}
By H\"{o}lder's inequality we have
\begin{align*}
&\|\mathrm{e}^{\psi(s,\cdot)}|v(s,\cdot)-\bar{v}(s,\cdot)|(|v(s,\cdot)|+|\bar{v}(s,\cdot)|)^{p-1}\|_{L^2}\\&\quad \leq \|\mathrm{e}^{(2-p)\psi(s,\cdot)}|v(s,\cdot)\!-\!\bar{v}(s,\cdot)|\|_{L^{2p}}\|\mathrm{e}^{(p-1)\psi(s,\cdot)}(|v(s,\cdot)|+|\bar{v}(s,\cdot)|)^{p-1}\|_{L^{\frac{2p}{p-1}}}.
\end{align*}

Let us estimate the two norms that appear at the right-hand side in the last inequality. Using Lemma \ref{2 lemma GN with weight} and $\psi\geq 0$, we obtain
\begin{align*}
\|\mathrm{e}^{(2-p)\psi(s,\cdot)}|v(s,\cdot)\!-\!\bar{v}(s,\cdot)|\|_{L^{2p}}\lesssim (1+s)^{(1-\theta(2p))}\|\mathrm{e}^{\psi(s,\cdot)}\nabla(v(s,\cdot)\!-\!\bar{v}(s,\cdot))\|_{L^{2}}
\end{align*}
and
\begin{align*}
\|&\mathrm{e}^{(p-1)\psi(s,\cdot)}(|v(s,\cdot)|+|\bar{v}(s,\cdot)|)^{p-1}\|_{L^{\frac{2p}{p-1}}}
\\&\qquad\lesssim \left(\|\mathrm{e}^{\psi(s,\cdot)}v(s,\cdot)\|_{L^{2p}}+\|\mathrm{e}^{\psi(s,\cdot)}\bar{v}(s,\cdot)\|_{L^{2p}}\right)^{p-1}\\ &\qquad\lesssim (1+s)^{(1-\theta(2p))(p-1)}\left(\|\mathrm{e}^{\psi(s,\cdot)}\nabla v(s,\cdot)\|_{L^{2}}+\|\mathrm{e}^{\psi(s,\cdot)}\nabla \bar{v}(s,\cdot)\|_{L^{2}}\right)^{p-1}.
\end{align*}
By \eqref{stima intermedia lemma loc esistenza} one gets
\begin{align*}
\|\mathrm{e}^{\psi(t,\cdot)}&w_t(t,\cdot)\|_{L^2}+\|\mathrm{e}^{\psi(t,\cdot)}\nabla w(t,\cdot)\|_{L^2} \\ &\quad \lesssim  \int_0^t (1+s)^{p(1-\theta(2p))}\|\mathrm{e}^{\psi(s,\cdot)}\nabla(v(s,\cdot)-\bar{v}(s,\cdot))\|_{L^{2}}\\& \quad \qquad \times\left(\|\mathrm{e}^{\psi(s,\cdot)}\nabla v(s,\cdot)\|_{L^{2}}+\|\mathrm{e}^{\psi(s,\cdot)}\nabla \bar{v}(s,\cdot)\|_{L^{2}}\right)^{p-1}ds\\ &\quad \lesssim\int_0^t (1+s)^{p(1-\theta(2p))}ds \|v-\bar{v}\|_T^\psi(\|v\|_T^\psi+\|\bar{v}\|_T^\psi)^{p-1}\\&\quad \lesssim T(1+T)^{p(1-\theta(2p))}K^{p-1} \|v-\bar{v}\|_T^\psi.
\end{align*} 
On the other hand 
\begin{align*}
\|\mathrm{e}^{\psi(t,\cdot)}w(t,\cdot)\|_{L^2}\lesssim m^{-1}(t)E_{\psi,w}(t)^{\frac{1}{2}}\lesssim T(1+T)^{p(1-\theta(2p))+1}K^{p-1} \|v-\bar{v}\|_T^\psi,
\end{align*} where the unexpressed multiplicative constants in this and in the previous chain of inequalities do not depend on $T$ and $K$. Summarizing 
\begin{align*}
\|w\|_T^\psi=\|\Phi(v)-\Phi(\bar{v})\|_T^\psi \lesssim T(1+T)^{p(1-\theta(2p))+1}K^{p-1} \|v-\bar{v}\|_T^\psi,
\end{align*} and then choosing $T>0$ small enough we have that $\Phi$ is a contraction.

Since in Lemma \ref{lemma E(T) banach} we proved that the space $\big(\mathrm{E}(T),\|\cdot\|_{T}^\psi\big)$ is a Banach space, then by Banach's fixed point theorem it is clear that our starting problem has a unique solution in $\mathcal{C}([0,T_m),H^1)\cap\mathcal{C}^1([0,T_m),L^2)$ with finite energy $E_{\psi,u}(t)$ for any $t\in [0,T_m)$.
Moreover $T_m<\infty$ implies the blow up of the energy for $T\to T_m^-$. Indeed if it was not so we would have a finite energy in a left neighborhood of $T_m$, and then repeating the same arguments seen for the particular case in which the initial conditions are taken for $t=0$, we could extend our solution.
\end{proof}

\begin{rmk} Differently from Theorem \ref{thm glob exi exp data}, Theorem \ref{thm local existence} is valid for any value of $\delta$ provided that $\mu_1>0$. 

This means that the Cauchy problem \eqref{CP semilinear 1} is locally in time well posed in the weighted evolution space $\mathcal{C}\big([0,T],H^1_{\psi(t,\cdot)}\big)\cap\mathcal{C}^1\big([0,T],L^2_{\psi(t,\cdot)}\big)$ for any $p>1$ such that $p\leq \frac{n}{n-2}$ when $n\geq 3$, independently from the value of $\delta$ and without any lower bound on $p$, due to the approach we are considering.
\end{rmk}

\section{Estimates for the linear problem} \label{Section Linear estimates}

In order to prove Theorem \ref{thm glob exi exp data}, we need to recall some known decay estimates for the solution of the linear parameter dependent Cauchy problem \eqref{linear parameter dependent CP}.

In the next propositions we can relax the assumptions on initial data, without consider the weighted energy spaces. Indeed, we may require just data in the classical energy spaces with additional $L^1$ regularity, namely
\begin{align*}
(u_0,u_1)\in (H^1\cap L^1)\times (L^2\cap L^1).
\end{align*} 

We will use the notation 
\begin{align*}
\mathcal{D}^\kappa :=(H^\kappa\cap L^1)\times (L^2\cap L^1) \qquad \mbox{for any} \,\, \kappa\in [0,1].
\end{align*} 
We put also $\mathcal{D}:= \mathcal{D}^1$.

When the data are taken at the initial time $s=0$, we have the following result.

\begin{prop}\label{Prop Lin Estim} Let $\mu_1>0$ and $\mu_2$ be nonnegative constants such that $\delta>0$. Let us consider $(u_0,u_1)\in \mathcal{D}$. Then for all $\kappa\in[0,1]$ the energy solution $u$ to \eqref{linear parameter dependent CP} with $s=0$ satisfies the decay estimates
\begin{align}\label{Lin Estim} \|u(t,\cdot)\|_{\dot{H}^\kappa}&\lesssim  \|(u_0,u_1)\|_{\mathcal{D}^\kappa} \begin{cases}(1+t)^{-\kappa-\frac{n+\mu_1}{2}+\frac{1+\sqrt{\delta}}{2}} & \mbox{if}\quad \kappa <\frac{1+\sqrt{\delta}-n}{2}, \\
(1+t)^{-\frac{\mu_1}{2}}\ell(t)  & \mbox{if}\quad \kappa =\frac{1+\sqrt{\delta}-n}{2}, \\
(1+t)^{-\frac{\mu_1}{2}} & \mbox{if}\quad  \kappa >\frac{1+\sqrt{\delta}-n}{2}, \end{cases}
\end{align} where $\ell(t):=1+(\log(1+t))^{\frac{1}{2}}$.
Moreover, $\| u_t(t,\cdot)\|_{L^2}$ satisfies the same decay estimates as $\|\nabla u(t,\cdot)\|_{L^2}$ which are obtained from \eqref{Lin Estim} after taking $\kappa=1$.
\end{prop}

Let us derive a result for the Cauchy problem \eqref{linear parameter dependent CP} in the case in which the first datum vanishes. According to Duhamel's principle, this type of result is necessary to estimate the integral term
\begin{align*}
\int_0^t E_1(t,s,x)\ast_{(x)} |u(s,x)|^p ds
\end{align*} that appears in the definition of the operator $N$.

\begin{prop}\label{Prop Lin Estim v_0=0 for t>s}
Let $\mu_1>0$ and $\mu_2$ be nonnegative constants such that $\delta>0$. Let us assume $u_0=0$ and $ u_1\in L^1\cap L^2$. Then the energy solution $u$ to \eqref{linear parameter dependent CP} satisfies for $t \geq s$ and $\kappa\in[0,1]$ the following estimates
\begin{equation}\begin{split}\label{Lin Estim v_0=0 for t>s}
\|u(t,\cdot)\|_{\dot{H}^\kappa}&\lesssim \Big(\|u_1\|_{L^1}+(1+s)^{\frac{n}{2}}\|u_1\|_{L^2}\Big)(1+s)^{\frac{1+\mu_1}{2}-\frac{\sqrt{\delta}}{2}}\\&\quad \times \begin{cases}(1+t)^{-\kappa-\frac{n+\mu_1}{2}+\frac{1+\sqrt{\delta}}{2}}  & \mbox{if}\quad \kappa <\frac{1+\sqrt{\delta}-n}{2}  ,\\ (1+t)^{-\frac{\mu_1}{2}}\widetilde{\ell}(t,s)
 & \mbox{if}\quad \kappa =\frac{1+\sqrt{\delta}-n}{2},\\ (1+t)^{-\frac{\mu_1}{2}}(1+s)^{-\kappa-\frac{n}{2}+\frac{1+\sqrt{\delta}}{2}} & \mbox{if}\quad \kappa >\frac{1+\sqrt{\delta}-n}{2},\end{cases}
\end{split}
\end{equation} where $\widetilde{\ell}(t,s):=1+\big(\log\big(\frac{1+t}{1+s}\big)\big)^{\frac{1}{2}}$.
 Moreover, $\| u_t(t,\cdot)\|_{L^2}$ satisfies the same decay estimates as $\|\nabla u(t,\cdot)\|_{L^2}$ which are obtained from \eqref{Lin Estim v_0=0 for t>s} after taking $\kappa=1$.
\end{prop}

The proofs of Propositions \ref{Prop Lin Estim} and \ref{Prop Lin Estim v_0=0 for t>s} are based on explicit representation formulas for the solution to \eqref{linear parameter dependent CP} and its time derivative.

In order to derive these representation formulas the scale-invariance of the linear equation in \eqref{linear parameter dependent CP} with respect to the \emph{hyperbolic scaling}
, together with the partial Fourier transform, is used. For a more precise presentation of the previous cited representation formulas one can see \cite{WiThe,Wi07}.

Concerning the proofs of Propositions \ref{Prop Lin Estim} and \ref{Prop Lin Estim v_0=0 for t>s}, we remark that they are a special case of Theorems 4.6 and 4.7 in \cite{PalRei17}.

\section{Global existence of small data solutions: proof of Theorem \ref{thm glob exi exp data}}\label{SDGE Section}

In order to prove Theorem \ref{thm glob exi exp data} we have first to prove the next preliminary lemma, which allows to estimate the \emph{weighted energy} of a local (in time) solution $u$ to \eqref{CP semilinear 1}
\begin{align*}
E_{\psi,u}(t)&=\frac{1}{2}\int_{\mathbb{R}^n}\mathrm{e}^{2\psi(t,x)}\left(|u_t(t,x)|^2+|\nabla u(t,x)|^2+m^2(t)|u(t,x)|^2\right)dx.
\end{align*}

\begin{lemma}\label{lemma before glob exist with exp weight} Let $n\geq 1$ and $\mu_1>0$, $\mu_2^2$ be the nonnegative constants. Let $p>1$ such that $p\leq \frac{n}{n-2}$ if $n\geq 3$. Let $(u_0,u_1)\in \mathcal{A}$. If $u=u(t,x)$ is a solution of 
\begin{align*}
\begin{cases}u_{tt} -\Delta u + \frac{\mu_1}{1+t} u_t + \frac{\mu_2^2}{(1+t)^2} u
=|u|^p,\quad (t,x)\in (0,T)\times\mathbb{R}^n, \\  u(0,x)=u_0(x), \quad x\in \mathbb{R}^n,\\  u_t(0,x)=u_1(x),\quad x\in \mathbb{R}^n,
\end{cases}
\end{align*} then the following energy estimate holds for any $t\in [0,T)$ and for arbitrary small $\eta>0$: 
\begin{equation}\label{lemma before glob exist with exp weight est fund}
E_{\psi,u}(t)\lesssim I^2_{\mu_1,\mu_2}+I^{p+1}_{\mu_1,\mu_2}+\bigg(\sup_{s\in[0,t]}(1+s)^\eta\|\mathrm{e}^{\left(\frac{2}{p+1}+\eta\right)\psi(s,\cdot)}u(s,\cdot)\|_{L^{p+1}}\bigg)^{p+1},
\end{equation} where 
\begin{align*}
I^2_{\mu_1,\mu_2}&=\int_{\mathbb{R}^n}\mathrm{e}^{\,\mu_1|x|^2}\left(|u_1(x)|^2+|\nabla u_0(x)|^2+\mu^2_2|u_0(x)|^2\right)dx.
\end{align*}
\end{lemma}

For the proof of this lemma we follow \cite[Lemma 2.1]{IkeTa05} and \cite[Lemma 5.2]{DabbLucRei13}.

\begin{proof}
We firstly prove that 
\begin{align}
E_{\psi,u}(t)&\lesssim I^2_{\mu_1,\mu_2}\!+I^{p+1}_{\mu_1,\mu_2}+\|\mathrm{e}^{\frac{2}{p+1}\psi(t,\cdot)}u(t,\cdot)\|_{L^{p+1}}^{p+1}\notag\\&\quad +\int_0^t\int_{\mathbb{R}^n}|\psi_t(s,x)|\mathrm{e}^{2\psi(s,x)}|u(s,x)|^{p+1}dxds.\label{lemma before glob exist with exp weight est 1}
\end{align}
Integrating the relation \eqref{fundamental equality 2} over $[0,t]\times\mathbb{R}^n$, we get immediately (after using the divergence theorem)
\begin{align*}
\mathscr{E}_{\psi,u}(t)\leq \mathscr{E}_{\psi,u}(0)-\frac{2}{p+1}\int_0^t\int_{\mathbb{R}^n}\psi_t(s,x)\mathrm{e}^{2\psi(s,x)}|u(s,x)|^{p}u(s,x)dxds,
\end{align*} where
\begin{align*}
\mathscr{E}_{\psi,u}(t)=E_{\psi,u}(t)-\frac{1}{p+1}\int_{\mathbb{R}^n}\mathrm{e}^{2\psi(t,x)}|u(t,x)|^{p}u(t,x)dx.
\end{align*}

 Consequently,
\begin{align*}
E_{\psi,u}(t)&\leq \mathscr{E}_{\psi,u}(0)+\frac{1}{p+1}\int_{\mathbb{R}^n}\mathrm{e}^{2\psi(t,x)}|u(t,x)|^{p}u(t,x)dx \\ & \quad -\frac{2}{p+1}\int_0^t\int_{\mathbb{R}^n}\psi_t(s,x)\mathrm{e}^{2\psi(s,x)}|u(s,x)|^{p}u(s,x)dxds\\
&\lesssim \mathscr{E}_{\psi,u}(0)+\|\mathrm{e}^{\frac{2}{p+1}\psi(t,\cdot)}u(t,\cdot)\|_{L^{p+1}}^{p+1}\\ & \quad +\int_0^t\int_{\mathbb{R}^n}|\psi_t(s,x)|\mathrm{e}^{2\psi(s,x)}|u(s,x)|^{p+1}dxds.
\end{align*}
 So, in order to prove \eqref{lemma before glob exist with exp weight est 1} we have just to show that $\mathscr{E}_{\psi,u}(0)\lesssim  I^2_{\mu_1,\mu_2}+I^{p+1}_{\mu_1,\mu_2}$. 
 
 Since 
\begin{align*}
\mathscr{E}_{\psi,u}(0)&=E_{\psi,u}(0)-\frac{1}{p+1}\int_{\mathbb{R}^n}\mathrm{e}^{\,\mu_1|x|^2}|u_0(x)|^p u_0(x)dx\\ &\lesssim I^2_{\mu_1,\mu_2}+\int_{\mathbb{R}^n}\mathrm{e}^{\,\mu_1|x|^2}|u_0(x)|^{p+1}dx,
\end{align*} we have to prove only that $$\displaystyle{\int_{\mathbb{R}^n}\mathrm{e}^{\,\mu_1|x|^2}|u_0(x)|^{p+1}dx\lesssim I^{p+1}_{\mu_1,\mu_2}}.$$

Because of $p+1<\frac{n}{n-2}+1<\frac{2n}{n-2}$ for $n\geq 3$ (no requirement for $n=1,2$), using the Sobolev embedding $H^1\hookrightarrow L^{p+1}$ we find
\begin{align*}
\int_{\mathbb{R}^n}&\mathrm{e}^{\,\mu_1|x|^2} |u_0(x)|^{p+1}dx
\lesssim \|\mathrm{e}^{\frac{\mu_1}{p+1}|x|^2}u_0(x)\|_{H^1}^{p+1}\\
&\quad=\left(\int_{\mathbb{R}^n}\mathrm{e}^{\frac{2\mu_1}{p+1}|x|^2}\left(\left(1+\left(\tfrac{2\mu_1}{p+1}\right)^2|x|^2\right)|u_0(x)|^{2}+|\nabla u_0(x)|^{2}\right)dx
\right)^{\frac{p+1}{2}}\\
&\quad \lesssim \left(\int_{\mathbb{R}^n}\mathrm{e}^{\,\mu_1|x|^2}\left(|u_0(x)|^{2}+|\nabla u_0(x)|^{2}\right)dx\right)^{\frac{p+1}{2}}\lesssim I_{\mu_1,\mu_2}^{p+1},
\end{align*} where in the second last inequality, we have used the fact that $p>1$ to get the estimate 
\begin{align*}
\left(1+\left(\tfrac{2\mu_1}{p+1}\right)^2|x|^2\right)\mathrm{e}^{\frac{2\mu_1}{p+1}|x|^2}\lesssim \mathrm{e}^{\,\mu_1|x|^2}.
\end{align*} 
So we proved \eqref{lemma before glob exist with exp weight est 1}.
Let us remark now that from the inequality $z\mathrm{e}^{-az}\leq C_a$ for any $a>0$  and the relation $\psi_t(s,x)=-\frac{2}{1+s}\psi(s,x)$, it follows
\begin{align*}
|\psi_t(s,x)|\mathrm{e}^{(2-\gamma(p+1))\psi(s,x)}=\tfrac{2}{1+s}\psi(s,x)\mathrm{e}^{-\eta(p+1)\psi(s,x)}\lesssim (1+s)^{-1},
\end{align*} with $\gamma=\frac{2}{p+1}+\eta$ and $\eta>0$. 

Therefore
\begin{align}
&\int_0^t\int_{\mathbb{R}^n}|\psi_t(s,x)|\mathrm{e}^{2\psi(s,x)}|u(s,x)|^{p+1}dx ds \notag\\ &\qquad\lesssim \int_0^t (1+s)^{-1}\int_{\mathbb{R}^n}\mathrm{e}^{\gamma (p+1)\psi(s,x)}|u(s,x)|^{p+1}dxds\notag \\
&\qquad \leq \sup_{s\in[0,t]}(1+s)^{\eta(p+1)}\|\mathrm{e}^{\gamma\psi(s,\cdot)}u(s,\cdot)\|_{L^{p+1}}^{p+1}\int_0^t (1+s)^{-1-\eta(p+1)}ds\notag\\
&\qquad \lesssim \bigg(\sup_{s\in[0,t]}(1+s)^{\eta}\|\mathrm{e}^{\gamma\psi(s,\cdot)}u(s,\cdot)\|_{L^{p+1}}\bigg)^{p+1}.\label{lemma before glob exist with exp weight est 2}
\end{align} 
Finally, since $\gamma>\frac{2}{p+1}$ and $\eta>0$
\begin{align}\label{last estimate local existence}
\|\mathrm{e}^{\frac{2}{p+1}\psi(t,\cdot)}u(t,\cdot)\|_{L^{p+1}}^{p+1}\leq \left((1+t)^{\eta}\|\mathrm{e}^{\gamma\psi(t,\cdot)}u(t,\cdot)\|_{L^{p+1}}\right)^{p+1}.
\end{align} Hence, combining \eqref{last estimate local existence}, \eqref{lemma before glob exist with exp weight est 1} and \eqref{lemma before glob exist with exp weight est 2} we get the desired estimate \eqref{lemma before glob exist with exp weight est fund}.
\end{proof} 

Combing the \emph{linear estimates} from Section \ref{Section Linear estimates} and Lemma \ref{lemma before glob exist with exp weight}, we can finally prove Theorem \ref{thm glob exi exp data}. Here we follow the main steps of \cite[Theorem 3]{D13}.

\begin{proof}[Proof of Theorem \ref{thm glob exi exp data}]
Let us assume that for any $\varepsilon_0>0$ there exists data satisfying \eqref{global existence eponential weight data cond} such that the solution $u\in\mathcal{C}\big([0,T_m),H^1_{\psi(t,\cdot)}\big)\cap \mathcal{C}^1\big([0,T_m),L^2_{\psi(t,\cdot)}\big)$ to the corresponding problem, whose existence is guaranteed by Theorem \ref{thm local existence}, is not global in time, that means $T_m<\infty$.

For any $T\in (0,T_m)$, we may define the Banach space 
\begin{align*}
X(T):=\mathcal{C}\big([0,T],H^1_{\psi(t,\cdot)}\big)\cap \mathcal{C}^1\big([0,T],L^2_{\psi(t,\cdot)}\big),
\end{align*} with the norm 
\begin{align*}
\|u&\|_{X(T)}:=\sup_{t\in[0,T]}\Big[(1+t)^{-1}\|u(t,\cdot)\|_{L^2_{\psi(t,\cdot)}}+(1+t)^{\frac{n}{2}+\frac{\mu_1}{2}-\frac{1}{2}-\frac{\sqrt{\delta}}{2}}\|u(t,\cdot)\|_{L^2}\\&\quad +\|(\nabla u,u_t)(t,\cdot)\|_{L^2_{\psi(t,\cdot)}} +(1+t)^{\frac{n}{2}+\frac{\mu_1}{2}+\frac{1}{2}-\frac{\sqrt{\delta}}{2}}\ell_\delta(t)^{-1}\|(\nabla u,u_t)(t,\cdot)\|_{L^2}\Big].
\end{align*} 

For simplicity of notations we will carry out the computations in the case $\delta>(n+1)^2$ only. However, in the logarithmic case $\delta=(n+1)^2$ no additional difficulty arises.

By Lemma \ref{lemma before glob exist with exp weight} it follows that
\begin{align}
(1+t)&^{-1}\|u(t,\cdot)\|_{L^2_{\psi(t,\cdot)}}+\|(\nabla u,u_t)(t,\cdot)\|_{L^2_{\psi(t,\cdot)}}\notag \\ &\lesssim   \varepsilon_0+\varepsilon_0^{\frac{p+1}{2}}+\bigg(\sup_{s\in[0,t]}(1+s)^\eta\|\mathrm{e}^{\left(\eta+\frac{2}{p+1}\right)\psi(s,\cdot)}u(s,\cdot)\|_{L^{p+1}}\bigg)^{\frac{p+1}{2}}.\label{global existence thm est 1}
\end{align} 
On the other hand from Theorems \ref{Prop Lin Estim} and \ref{Prop Lin Estim v_0=0 for t>s} we have
\begin{align}
&(1+t)^{\frac{n}{2}+\frac{\mu_1}{2}-\frac{1}{2}-\frac{\sqrt{\delta}}{2}}\left(\|u(t,\cdot)\|_{L^2}+(1+t)\|(\nabla u,u_t)(t,\cdot)\|_{L^2}\right)\notag\\&\qquad\lesssim \varepsilon_0+\int_0^t (1+s)^{\frac{1+\mu_1}{2}-\frac{\sqrt{\delta}}{2}}\Big(\|\,|u(s,\cdot)|^p\|_{L^1}+(1+s)^\frac{n}{2}\|\,|u(s,\cdot)|^p\|_{L^2}\Big)ds.\label{global existence thm est 1,5}
\end{align}
 Using \eqref{L1-L2 weight est} and \eqref{L2-L2 weight est} for $f=|u(s,\cdot)|^p$, we get immediately that
\begin{align}\label{global existence thm est 2}
\|\,|u(s,\cdot)|^p\|_{L^1}+(1+s)^\frac{n}{2}\|\,|u(s,\cdot)|^p\|_{L^2}\lesssim (1+s)^\frac{n}{2}\|\mathrm{e}^{\eta\psi(s,\cdot)}u(s,\cdot)\|_{L^{2p}}^p.
\end{align} 

Now we apply Lemma \ref{2 lemma GN with weight} in order to estimate the terms which appear in the right hand side of  \eqref{global existence thm est 1} and \eqref{global existence thm est 2}. 

Being $2<p+1<2p$ and $p+1<2p\leq \frac{2n}{n-2}$ if $n\geq 3$, we find that $\theta(p+1),\theta(2p)\in (0,1]$. Moreover, we can choose $\eta>0$ sufficiently small such that $\eta+\frac{2}{p+1}<1$. Hence, by Lemma \ref{2 lemma GN with weight} we obtain
\begin{align*}
\|\mathrm{e}^{\left(\eta+\frac{2}{p+1}\right)\psi(s,\cdot)}&u(s,\cdot)\|_{L^{p+1}} \\ &\lesssim (1+s)^{1-\theta(p+1)}\|\nabla u(s,\cdot)\|_{L^2}^{1-\left(\eta+\frac{2}{p+1}\right)}\|\mathrm{e}^{\psi(s,x)}\nabla u(s,\cdot)\|_{L^2}^{\eta+\frac{2}{p+1}}\\ &\lesssim (1+s)^{1-\theta(p+1)-\left(1-\left(\eta+\frac{2}{p+1}\right)\right)\left(\frac{n}{2}+\frac{\mu_1}{2}+\frac{1}{2}-\frac{\sqrt{\delta}}{2}\right)}\|u\|_{X(T)},
\end{align*} and  
\begin{align*}
\|\mathrm{e}^{\eta\psi(s,\cdot)}u(s,\cdot)\|_{L^{2p}}&\lesssim (1+s)^{1-\theta(2p)}\|\nabla u(s,\cdot)\|_{L^2}^{1-\eta}\|\mathrm{e}^{\psi(s,x)}\nabla u(s,\cdot)\|_{L^2}^{\eta}\\ &\lesssim (1+s)^{1-\theta(2p)-(1-\eta)\left(\frac{n}{2}+\frac{\mu_1}{2}+\frac{1}{2}-\frac{\sqrt{\delta}}{2}\right)}\|u\|_{X(T)}.
\end{align*} 
Then
\begin{align*}
\sup_{s\in[0,t]}&(1+s)^\eta\|\mathrm{e}^{\left(\eta+\frac{2}{p+1}\right)\psi(s,\cdot)}u(s,\cdot)\|_{L^{p+1}} \\ &\lesssim\|u\|_{X(T)} \sup_{s\in[0,t]}(1+s)^{\eta+1-\theta(p+1)-\left(1-\left(\eta+\frac{2}{p+1}\right)\right)\left(\frac{n}{2}+\frac{\mu_1}{2}+\frac{1}{2}-\frac{\sqrt{\delta}}{2}\right)},
\end{align*} and 
\begin{align*}
\int_0^t &(1+s)^{\frac{1+\mu_1}{2}-\frac{\sqrt{\delta}}{2}}\Big(\|\,|u(s,\cdot)|^p\|_{L^1}+(1+s)^\frac{n}{2}\|\,|u(s,\cdot)|^p\|_{L^2}\Big)ds\\&\quad\lesssim \|u\|_{X(T)}^p \int_0^t(1+s)^{\frac{1+\mu_1}{2}-\frac{\sqrt{\delta}}{2}+\frac{n}{2}+\big(1-\theta(2p)-(1-\eta)\big(\frac{n}{2}+\frac{\mu_1}{2}+\frac{1}{2}-\frac{\sqrt{\delta}}{2}\big)\big)p}ds .
\end{align*}
Being $p>p_{\Fuj}\Big(n+\frac{\mu_1-1}{2}-\frac{\sqrt{\delta}}{2}\Big)$, we can find $\eta>0$ such that 
\begin{align*}
&\eta+1-\theta(p+1)-\left(1-\left(\eta+\tfrac{2}{p+1}\right)\right)\left(\tfrac{n}{2}+\tfrac{\mu_1}{2}+\tfrac{1}{2}-\tfrac{\sqrt{\delta}}{2}\right)<0,\\
&\tfrac{1+\mu_1}{2}-\tfrac{\sqrt{\delta}}{2}+\tfrac{n}{2}+\left(1-\theta(2p)-(1-\eta)\left(\tfrac{n}{2}+\tfrac{\mu_1}{2}+\tfrac{1}{2}-\tfrac{\sqrt{\delta}}{2}\right)\right)p<-1.
\end{align*}

 Therefore from \eqref{global existence thm est 1} and \eqref{global existence thm est 1,5} it follows 
\begin{align}\label{est for the norm of u in X(T)}
\|u\|_{X(T)}\leq C_0\Big(\varepsilon_0+\varepsilon_0^{\frac{p+1}{2}}\Big)+C_1\|u\|_{X(T)}^{\frac{p+1}{2}}+C_2\|u\|_{X(T)}^p
\end{align} for some constants $C_0,C_1,C_2>0$. If $\varepsilon_0$ is small enough, then from the last inequality we get that $\|u\|_{X(T)}$ is uniformly bounded, more precisely $$\|u\|_{X(T)}\lesssim \varepsilon_0$$ for $T\in(0,T_m)$. 

Let us show how to prove the last property.
Define the function
\begin{align*}
\varphi(x)=x-C_1x^{\frac{p+1}{2}}-C_2x^{p}.
\end{align*} 
We have $\varphi(0)=0$ and $\varphi'(0)=1$. Furthermore, $\varphi(x)\leq x$ for any $x\geq 0$ and there exists $\bar{x}\geq 0$ such that $\varphi'(x)\geq \frac{1}{2}$ for any $x\in[0,\overline{x}]$.  Consequently, $\varphi$ is a strictly increasing function on $[0,\overline{x}]$ and 
\begin{equation}\label{phi inequality}
\varphi(x)\leq x \leq 2\varphi(x)
\end{equation}
for all $[0,\overline{x}]$. Let us define 
\begin{align*}
\varepsilon_0=\min\left\{1,\tfrac{\overline{x}}{2},\tfrac{\overline{x}}{4C_0}\right\}.
\end{align*} 

If $\|(u_0,u_1)\|_{\mathcal{A}}=\varepsilon$ for some $\varepsilon\in(0,\varepsilon_0]$, then 
\begin{align}\label{phi (||u||_X(0)) estimate}
\|u\|_{X(0)}\leq 2\|(u_0,u_1)\|_{\mathcal{A}}=2\varepsilon\leq \overline{x}.
\end{align} 
Since $\varphi$ is strictly increasing on $[0,\overline{x}]$, it follows:
\begin{equation*}
\varphi\big(\|u\|_{X(0)}\big)\leq \varphi(\overline{x}).
\end{equation*}
Thanks to \eqref{est for the norm of u in X(T)} we get 
\begin{equation}\label{phi (||u||_X(T)) estimate}
\varphi\big(\|u\|_{X(T)}\big)\leq C_0\Big(\varepsilon_0+\varepsilon_0^{\frac{p+1}{2}}\Big)\leq 2C_0\varepsilon_0.
\end{equation}

Hence, we find
\begin{equation}\label{phi (||u||_X(T)) estimate 2}
\varphi\big(\|u\|_{X(T)}\big)\leq \tfrac{\overline{x}}{2}\leq \varphi(\overline{x})
\end{equation} for any $T\in[0,T_m)$. Therefore, since $\|u\|_{X(T)}$ is a continuous function for $T\in(0,T_m)$ and using once again the fact that $\varphi$ is strictly increasing on $[0,\overline{x}]$, if we combine \eqref{phi (||u||_X(0)) estimate} and \eqref{phi (||u||_X(T)) estimate 2}, then it follows immediately that 
\begin{align*}
\|u\|_{X(T)}\leq \overline{x}
\end{align*} for any $T\in(0,T_m)$. Using \eqref{phi inequality}, from the above inequality and \eqref{phi (||u||_X(T)) estimate} we get the desired inequality
\begin{align*}
\|u\|_{X(T)}\leq 2 \varphi\big(\|u\|_{X(T)}\big)\leq 4C_0\varepsilon_0.
\end{align*}
 Therefore, it holds
\begin{align*}
\limsup_{T\to T_m^-}\left(\|u(t,\cdot)\|_{H^1_{\psi(t,\cdot)}}+\|u_t(t,\cdot)\|_{L^2_{\psi(t,\cdot)}}\right)\lesssim \limsup_{T\to T_m^-}\|u\|_{X(T)}\lesssim \varepsilon_0 .
\end{align*} 

Nevertheless, this is impossible according to the last part of Theorem \ref{thm local existence}, so  $T_m=\infty$, that is $u$, has to be a global solution. The estimates of the statement follows by the relation $\|u\|_{X(T)}\lesssim \varepsilon_0$, which holds uniformly with respect to $T$. 
\end{proof}

\section{Blow-up: proof of Theorem \ref{thm blow up}}\label{B-U Section}

Let us state a preliminary lemma that we are going to use in the proof of Theorem \ref{thm blow up}. 

\begin{lemma}\label{lemma ODI} Let $K_0,K_1>0$, $\alpha\geq -2$ and $p>1$. Let $F\in \mathcal{C}^2([0,T))$ satisfying
\begin{align}
&\ddot{F}(t)+K_0 (1+t)^{-1} \dot{F}(t)\geq K_1 (1+t)^\alpha |F(t)|^p \qquad \mbox{for any} \,\, t\in [0,T). \label{ordinary differential inequality for F}
\end{align}  If 
$F(0)>0$, $\dot{F}(0)>0$, then $T<\infty$, that is,  $F$ blows up in finite time.
\end{lemma}

The previous lemma is a variation of \cite[Proposition 3.1]{TodYor01} that fits finely to our equation.

\begin{proof}
Let us consider the solution $G(t)$ of the ordinary differential equation
\begin{align}
\dot{G}(t)=\nu (1+t)^{\alpha+1}(G(t))^{\frac{p+1}{2}},\label{ODE for G}
\end{align} with initial condition $G(0)=F(0)>0$, where $\nu$ is a positive constant suitably small that we will fix afterwards.
By using separation of variables, we get 
\begin{align*}
G(t)^{-\frac{p-1}{2}}&=G(0)^{-\frac{p-1}{2}}-\tfrac{2\nu}{(p-1)(\alpha+2)}\big((1+t)^{\alpha+2}-1\big) & \mbox{if} \,\, \alpha>-2, \\
G(t)^{-\frac{p-1}{2}}&=G(0)^{-\frac{p-1}{2}}-\tfrac{2\nu}{p-1}\log(1+t)& \mbox{if} \,\, \alpha=-2 .
\end{align*}  In both cases the right-hand side tends to $0^+$ as $t\to T_0^-$, where $T_0$ denotes the life span of $G$
\begin{align*}
T_0:=\begin{cases} \Big(\frac{(p-1)(\alpha+2)}{2\nu}G(0)^{-\frac{p-1}{2}}+1\Big)^{\frac{1}{\alpha+2}}-1 & \mbox{if} \,\, \alpha> -2, \\ 
\mathrm{e}^{\frac{p-1}{2\nu}G(0)^{-\frac{p-1}{2}}} -1& \mbox{if} \,\, \alpha= -2. \end{cases}
\end{align*}

 Consequently, $G(t)$ blows up as $t\to T_0^-$. If we prove that is possible to control from below $F(t)$ with $G(t)$, then necessarily $F(t)$ blows up in finite time as well. 
 
 For the second derivative of $G$ it holds
\begin{align*}
\ddot{G}(t)
= \tfrac{\nu^2(p+1)}{2} (1+t)^{2(\alpha+1)}(G(t))^{p}+\nu (\alpha+1)(1+t)^\alpha(G(t))^{\frac{p+1}{2}}.
\end{align*} Thus,
\begin{align}
&\ddot{G}(t)+K_0(1+t)^{-1}\dot{G}(t) \notag
\\&\leq \tfrac{\nu^2(p+1)}{2} (1+t)^{2(\alpha+1)}(G(t))^{p}+\nu (\alpha+1+K_0) (1+t)^\alpha(G(t))^{p}(G(0))^{-\frac{p-1}{2}} \notag \\
&\leq \nu\Big(\tfrac{\nu(p+1)}{2} +(\alpha+1+K_0)(G(0))^{-\frac{p-1}{2}} \Big) (1+t)^\alpha(G(t))^{p}, \label{ordinary differential inequality for G}
\end{align} where in the first inequality we used $(G(t))^{-\frac{p-1}{2}}\leq (G(0))^{-\frac{p-1}{2}}$, while in second one we employed $\alpha\geq -2$.
We may choose $\nu>0$ such that
\begin{align*}
\begin{cases}
\nu\Big(\tfrac{\nu(p+1)}{2} +(\alpha+1+K_0)(G(0))^{-\frac{p-1}{2}} \Big) <K_1, \\  \dot{G}(0)=\nu (G(0))^{\frac{p+1}{2}}<\dot{F}(0).
\end{cases}
\end{align*}
Hence,
\begin{align*}
\begin{cases} \ddot{G}(t)+K_0(1+t)^{-1}\dot{G}(t)\leq K_1(1+t)^{\alpha}(G(t))^{p},  \\ 
G(0)=F(0)\, , \,\,  \dot{G}(0)<\dot{F}(0).
\end{cases}
\end{align*} 

The next step is to prove that $F(t)\geq G(t)$ for any $t\in [0,T_0)$, which implies the blow-up of $F$ in finite time.

If the life-span of $F(t)$ is strictly less than $T_0$, we are done. It remains to consider the case in which $F(t)$ is defined for any $t<T_0$.

 Because of the continuity of $\dot{F}$ and $\dot{G}$ the inequality $\dot{F}(0)>\dot{G}(0)$ implies $\dot{F}(t)>\dot{G}(t)$ at least for $t$ in a right neighborhood of $0$. Let us define 
\begin{align*}
t_0:=\sup\big\{t\in (0,T_0]: \dot{F}(\tau)>\dot{G}(\tau) \,\, \mbox{for any} \,\, \tau \in [0,t)\big\}.
\end{align*}
 If we prove that $t_0=T_0$, then $F-G$ is strictly increasing on $(0,T_0)$ and, in particular, $F(0)=G(0)$ implies $F(t)>G(t)$ for any $t\in (0,T_0)$ which concludes the proof.

By contradiction we assume that $t_0<T_0$. Therefore, $F-G$ is strictly increasing on $(0,t_0)$, and consequently $F(t)>G(t)$ for all $t\in (0,t_0)$.

 Moreover, it holds $F(t_0)>G(t_0)$, otherwise since we had $F(0)=G(0)$ and $F(t_0)=G(t_0)$ by the mean value theorem we would find $t_1\in (0,t_0)$ such that $\dot{F}(t_1)=\dot{G}(t_1)$, but this would be impossible according to the definition of $t_0$. Then $F(t_0)>G(t_0)$ and $\dot{F}(t_0)=\dot{G}(t_0)$. 
 
 Subtracting \eqref{ordinary differential inequality for G} from \eqref{ordinary differential inequality for F} we obtain
\begin{align}
\big(\ddot{F}(t)-\ddot{G}(t)\big)+& K_0(1+t)^{-1}\big(\dot{F}(t)-\dot{G}(t)\big)\notag\\ & \geq K_1(1+t)^\alpha \big[(F(t))^p-(G(t))^p\big]\geq 0  \label{ordinary diffenetial inequality for F-G}
\end{align}for any $t\in [0,t_0]$.

Multiplying both sides of inequality \eqref{ordinary diffenetial inequality for F-G} by $(1+t)^{K_0}$, we have
\begin{align*}
\frac{d}{dt}\big[(1+t)^{K_0}\big(\dot{F}(t)-\dot{G}(t)\big)\big]\geq 0.
\end{align*} Integrating the above relation on $[0,t_0]$, it follows
\begin{align}
(1+t_0)^{K_0}\big(\dot{F}(t_0)-\dot{G}(t_0)\big)\geq \dot{F}(0)-\dot{G}(0)>0,
\end{align} and hence $\dot{F}(t_0)>\dot{G}(t_0)$, which is impossible. This concludes the proof.
\end{proof}

\begin{proof}[Proof of Theorem \ref{thm blow up}]
Let us carry out the transformation
\begin{align*}
v(t,x):=(1+t)^{\frac{\mu_1-1}{2}-\frac{\sqrt{\delta}}{2}}u(t,x).
\end{align*} 
Thus, $u$ is a solution to \eqref{CP semilinear 1} if and only if $v$ is a solution to the following Cauchy problem:
\begin{align*}
\begin{cases} 
v_{tt}-\Delta v+ \frac{1+\sqrt{\delta}}{1+t}v_t=(1+t)^{-\frac{\mu_1-1-\sqrt{\delta}}{2}(p-1)}|v|^p, \\
v(0,x)=v_0(x), \\
v_t(0,x)=v_1(x).
\end{cases}
\end{align*}
where $v_0(x)=u_0(x)$ and $v_1(x)=u_1(x)+\Big(\tfrac{\mu_1-1}{2}-\tfrac{\sqrt{\delta}}{2}\Big)u_0(x)$.

 Let us choose $R>0$ such that $\supp v_0, \supp v_1 \subset B_R$, where $B_R$ is the ball centered in the origin with radius $R$. Thanks to the finite speed of propagation of $v$ we  have $\supp v(t,\cdot)\subset B_{R+t}$.
Define 
\begin{align*}
F(t):=\int_{\mathbb{R}^n} v(t,x) dx.
\end{align*}
Then we obtain 
\begin{align*}
\ddot{F}(t)&=\int_{\mathbb{R}^n} v_{tt}(t,x) dx=\int_{\mathbb{R}^n} \Delta v(t,x) dx-\tfrac{1+\sqrt{\delta}}{1+t}\int_{\mathbb{R}^n}  v_t(t,x) dx\\&\quad +(1+t)^{-\frac{\mu_1-1-\sqrt{\delta}}{2}(p-1)}\int_{\mathbb{R}^n} |v(t,x)|^p dx\\
&=-\tfrac{1+\sqrt{\delta}}{1+t}\dot{F}(t)+(1+t)^{-\frac{\mu_1-1-\sqrt{\delta}}{2}(p-1)}\int_{B_{t+R}} |v(t,x)|^p dx,
\end{align*} where in the last equality we use the divergence theorem and the fact that $\Delta v$ is compactly supported. Jensen's inequality implies that
\begin{align}\label{ODI F integral}
\ddot{F}(t)+\tfrac{1+\sqrt{\delta}}{1+t}\dot{F}(t)\gtrsim (1+t)^{-\big(n+\frac{\mu_1-1}{2}-\frac{\sqrt{\delta}}{2}\big)(p-1)}|F(t)|^p.
\end{align} The assumptions on $u_0,u_1$ guarantee that 
\begin{align*}
F(0)=\int_{\mathbb{R}^n} v_0(x) dx\,>0 \,\,\,\mbox{and} \,\,\, \dot{F}(0)=\int_{\mathbb{R}^n} v_1(x) dx>\,0.
\end{align*}

 Concluding we can apply Lemma \ref{lemma ODI}, since \eqref{upper bound p blow-up} corresponds to the fact that the power of $(1+t)$ in the right-hand side of \eqref{ODI F integral} is less than or equal to $-2$.
\end{proof}

\section{Conclusion}

Combining results obtained in Theorems \ref{thm glob exi exp data} and \ref{thm blow up}, we find that the critical exponent in every space dimension for the semi-linear Cauchy problem \eqref{CP semilinear 1} is $p_{\Fuj}\big(n+\frac{\mu_1-1}{2}-\frac{\sqrt{\delta}}{2}\big)$, provided that $\mu_1>0$ and $\mu_2^2$ satisfy the condition $\delta\geq (n+1)^2$.

Of course this result is consistent with that one proved in \cite{NunPalRei16}, according to our original purpose, that was exactly to extend such result to any spatial dimension for the same range of $\delta$.

Let us point out that one can not expect that this kind of approach with exponentially weighted energy spaces works optimally for any possible choice of coefficients $\mu_1$ and $\mu_2^2$, although Theorem \ref{thm local existence} holds independently of $\delta$. 

Indeed, as we have seen in our treatment, in some sense we consider the case in which the considered equation presents some \emph{parabolic effect}. Nevertheless, a deeper analysis of the linear equation related to \eqref{CP semilinear 1} shows that the interplay between the coefficients $\mu_1$ and $\mu_2^2$, described through the quantity $\delta$, may cause extremely different qualitative effects for different values of  $\delta$.
A more precise, but apparently still incomplete, classification of these different possible  effects is given in \cite{PalThe}.

Finally, we mention that in \cite{PalRei17} different ideas are applied in the study of \eqref{CP semilinear 1}. More precisely, in that paper the influence of higher regularity for data on the regularity of the solution is studied, analyzing the way in which the exponent of the non-linearity has to be chosen in order to prove the well-posedness of \eqref{CP semilinear 1} in Sobolev spaces with higher regularity. The reason to consider higher regularity is that in such a way we can either enlarge the upper bound for $p$ when $n\geq 3$ or even remove it for suitabely large regularity.

 Nevertheless, we have to pay for this choice by requiring a greater lower bound for $p$. For sake of completeness, let us present the statement of the above cited result and the main toolts from harmonic analysis that are necessary in the proof. For a deeper analysis of these tools and for the proof of the global existence result the interested reader can see \cite{PalRei17}.
\begin{prop} 
 Let us assume $\sigma\geq 1$. Let $\mu_1>1$ and $\mu_2$ be nonnegative constants such that $\delta \geq ( n+2\sigma-1)^2$. Let us choose $p\geq 2$ and $p>\lceil \sigma\rceil=:\min \{k\in \mathbb{Z}:\sigma\leq k\}$
 such that
\begin{align*}
p& \leq  1+\tfrac{2}{n-2\sigma} \quad \mbox{if}\,\, n > 2\sigma, \\
p&>p_{\Fuj}\left(n+\tfrac{\mu_1-1}{2}-\tfrac{\sqrt{\delta}}{2}\right).
\end{align*}
 Then there exists a constant $\varepsilon_0>0$ such that for all $$(u_0,u_1) \in \mathcal{D}^\sigma:= (L^1\cap H^\sigma )\times (L^1\cap H^{\sigma-1})$$ with
$ \| (u_0,u_1) \|_{\mathcal{D}^\sigma} \leq \varepsilon_0$ there is a uniquely determined energy solution $$u\in\mathcal{C}([0,\infty), H^\sigma) \cap \mathcal{C}^1([0,\infty), H^{\sigma-1})$$ to \eqref{CP semilinear 1}.
Moreover, the solution satisfies the decay estimates
\begin{align*}
\|u(t,\cdot)\|_{L^2}& \lesssim (1+t)^{-\frac{\mu_1}{2}+\frac{1+\sqrt{\delta}}{2}-\frac{n}{2}} \|(u_0,u_1)\|_{\mathcal{D}^\sigma}, \\
\|u_t(t,\cdot)\|_{L^2}& \lesssim (1+t)^{-\frac{\mu_1}{2}+\frac{1+\sqrt{\delta}}{2}-1-\frac{n}{2}}\|(u_0,u_1)\|_{\mathcal{D}^\sigma},\\
\| u(t,\cdot)\|_{\dot{H}^\sigma}& \lesssim (1+t)^{-\frac{\mu_1}{2}+\frac{1+\sqrt{\delta}}{2}-\sigma-\frac{n}{2}}\ell_{\delta,\sigma}(t) \|(u_0,u_1)\|_{\mathcal{D}^\sigma}, \\
\|u_t(t,\cdot)\|_{\dot{H}^{\sigma-1}}& \lesssim (1+t)^{-\frac{\mu_1}{2}+\frac{1+\sqrt{\delta}}{2}-\sigma-\frac{n}{2}}\ell_{\delta,\sigma}(t) \|(u_0,u_1)\|_{\mathcal{D}^\sigma},
\end{align*} where
\begin{align*}
\ell_{\delta,\sigma}(t)=\begin{cases}1 &\mbox{if}\quad \delta>\left( n+2\sigma-1\right)^2, \\1+\left(\log(1+t)\right)^\frac{1}{2} &\mbox{if}\quad \delta=\left( n+2\sigma-1\right)^2. \end{cases}
\end{align*}
\end{prop}

As we said before, in order to prove this result the following tools from harmonic analysis are used. 

\begin{lemma}\label{Thm FGNI}
Let $1<p,p_0,p_1<\infty$ and $\kappa\in [0,\sigma)$. Then it holds the following fractional Gagliardo-Nirenberg inequality for all $u\in L^{p_0}\cap \dot{H}^\sigma_{p_1}$:
\begin{align*}
\|u\|_{\dot{H}^{\kappa}_p}\lesssim \|u\|_{L^{p_0}}^{1-\theta}\|u\|_{\dot{H}^{\sigma}_{p_1}}^\theta,
\end{align*} where $\theta=\theta_{\kappa,\sigma}(p,p_0,p_1)=\frac{\frac{1}{p_0}-\frac{1}{p}
+\frac{\kappa}{n}}{\frac{1}{p_0}-\frac{1}{p_1}+\frac{\sigma}{n}}$ and $\frac{\kappa}{\sigma}\leq \theta\leq 1$.
\end{lemma}

\begin{lemma}\label{thm FLR} Let us assume $\sigma>0$ and $1\leq r \leq \infty, 1<p_1,p_2,q_1,q_2\leq \infty$ satisfying the relation $$\tfrac{1}{r}=\tfrac{1}{p_1}+\tfrac{1}{p_2}=\tfrac{1}{q_1}+\tfrac{1}{q_2}.$$ Then it holds the following fractional Leibniz rule:
\begin{align*}
\|\,|D|^\sigma(u \,v)\|_{L^r}\lesssim \|\,|D|^\sigma u\|_{L^{p_1}}\|v\|_{L^{p_2}}+\|u\|_{L^{q_1}}\|\,|D|^\sigma v\|_{L^{q_2}}
\end{align*}  for any $u$ and $v$ such that the norms of the right-hand side exist.
\end{lemma}

\begin{lemma}\label{Thm FCR} Let us denote by $F(u)$ one of the functions $|u|^p, \pm |u|^{p-1}u$.
 Let us consider $\sigma>0$ such that $p>\lceil \sigma \rceil$
 and $1<r,r_1,r_2<\infty$ satisfying the condition
\begin{equation*}
\frac{1}{r}=\frac{p-1}{r_1}+\frac{1}{r_2}.
\end{equation*}
 Then it holds the following fractional chain rule:
\begin{align*}
\|\,|D|^{\sigma} F(u)\|_{L^r}\lesssim \|u\|_{L^{r_1}}^{p-1}\|\,|D|^{\sigma} u\|_{L^{r_2}}
\end{align*}
 for any $u\in  L^{r_1}\cap \dot{H}^{\sigma}_{r_2}$.
\end{lemma}

While the proof of Lemmas \ref{Thm FGNI} and \ref{thm FLR} can be found in \cite{HMOW11} and \cite{Gra15}, respectively, the result stated in Lemma \ref{Thm FCR}, whose proof can be found in \cite{PalRei17}, up to the knowledge of the author, is the first generalization to general order $\sigma>0$ and arbitrary spatial dimension $n$ of a classic result originally proved in \cite{Chri91,Sta95}.


\subsection*{Acknowledgment}

The author is member of the Gruppo Nazionale per L'Analisi Matematica, la Probabilit\`{a} e le loro Applicazioni (GNAMPA) of the Instituto Nazionale di Alta Matematica (INdAM). The author thanks Marcello D'Abbicco (University of Bari) for his numerous hints and Michael Reissig (TU Bergakademie Freiberg) for several fruitful discussions on the subject and for his helpful suggestions in the preparation of this paper.

\end{document}